\documentclass[12pt]{amsart}
\usepackage{lipsum}
\usepackage{amssymb}
\usepackage[all]{xy}
\usepackage{enumerate}
\usepackage{amsthm}
\usepackage{tikz-cd}
\usepackage{amsmath}
\usepackage[mathscr]{euscript}
\usepackage[utf8]{inputenc}
\usepackage[english]{babel}
\usepackage{hyperref}
\hypersetup{
    colorlinks=true,
    linkcolor=blue,
    citecolor=red,
    filecolor=red,
    urlcolor=violet,
}

\makeatletter
\@namedef{subjclassname@2020}{%
  \textup{2020} Mathematics Subject Classification}
\makeatother

\newtheorem{thm}{Theorem}[section]
\newtheorem{lem}[thm]{Lemma}
\newtheorem{prop}[thm]{Proposition}
\newtheorem{defi}[thm]{Definition}
\newtheorem{qn}[thm]{Question}

\newtheorem{corl}[thm]{Corollary}
\newtheorem{xrem}[thm]{Remark}
\newtheorem{exm}[thm]{Example}

\DeclareMathOperator{\Nef}{{Nef}}

\DeclareMathOperator{\rk}{{rk}}
\DeclareMathOperator{\Mov}{{Mov}}

\DeclareMathOperator{\Pic}{{Pic}}
\DeclareMathOperator{\Sym}{{Sym}}
\DeclareMathOperator{\Eff}{{Eff}}
\DeclareMathOperator{\End}{{\mathcal{E}nd}}
\DeclareMathOperator{\Div}{{Div}}
\DeclareMathOperator{\Proj}{{Proj}}
\DeclareMathOperator{\Bl}{{Bl}}

\DeclareMathOperator{\vol}{{vol}}

\begin{document}
\baselineskip=15pt

\author{Snehajit Misra}
\author{Anoop Singh} 

\address{Indraprastha Institute of Information Technology, Delhi, Okhla Industrial Estate, Phase
III, New Delhi, Delhi 110020, India.}
\email[Snehajit Misra]{misra08@gmail.com}

\address{Department of Mathematical Sciences\\ Indian Institute of Technology(BHU), Varanasi 211005, India.}
\email[Anoop Singh]{anoopsingh.mat@iitbhu.ac.in}

\begin{abstract}
 In this article, we compute the pseudo-effective cones of various projective bundles $\mathbb{P}_X(E)$  over higher dimensional varieties $X$ under some assumptions on $X$ as well as on the vector bundle $E$. We also compute  the volume function on fibre product  $\mathbb{P}(E)\times_C\mathbb{P}(F)$ of two projective bundles over a smooth irreducible complex projective curve $C$.
 In particular, we show that the volume function on the fiber product of two ruled surfaces is of polynomial type.
\end{abstract}

\title{On pseudo-effective cones of projective bundles and volume function}
\maketitle
\section{Introduction}
  
 A fundamental invariant of a smooth irreducible projective variety $X$ is the closure of the cone generated by effective divisors on $X$, called the pseudo-effective cone of $X$ and denoted by $\overline{\Eff}^1(X)$. This invariant
plays  a
 very crucial role in the understanding the geometry of the higher dimensional projective varieties. The interior of $\overline{\Eff}^1(X)$ is called Big cone of $X$, and its elements are called  big divisors.

  Miyaoka \cite{Miy87} computed the cones of nef and
effective divisors of a projective bundle $\mathbb{P}_C(E)$ over a smooth irreducible projective curve $C$ in characteristics 0, where $E$ is any rank $r$ vector bundle on $C$, and gives a numerical criterion for 
semi-stability of $E$ in terms of nefness of the normalized hyperplane class $\lambda_E$ on $\mathbb{P}_C(E)$. More generally, in \cite{Fu}, the cone of effective $k$-cycles in $\mathbb{P}_C(E)$
over an irreducible curve $C$ is described in terms of the numerical
data appearing in  Harder-Narasimhan filtration of the bundle $E$, which generalizes Miyaoka's result in the semistable case. However, in  these cases, the Picard number of the projective bundle $\mathbb{P}_C(E)$ is 2, and hence the pseudo-effective cones  of divisors are finite polyhedra  generated by two extremal rays in
a two dimensional real vector space. When the Picard number is at least 3, the pseudo-effective cones are not necessarily polyhedra. Recently, the pseudo-effective cones of divisors of the fiber product $\mathbb{P}_C(F)\times_C\mathbb{P}_C(E)$ of two projective bundles over a smooth irreducible complex projective curve $C$ which has Picard number 3 are computed in \cite{M21}. It turns out that the pseudo-effective cones $\overline{\Eff}^1\bigl(\mathbb{P}_C(F)\times_C\mathbb{P}_C(E)\bigr)$ are finite polyhedra.
In this article, our first observation is the following.
\begin{thm}\label{thm1.1}
 Let $E$ be a slope semistable vector bundle of rank $r$ on a smooth irreducible complex projective variety $X$ such that $c_2\bigl(\End(E)\bigr) = 2rc_2(E)-(r-1)c^2_1(E) = 0 \in H^4(X,\mathbb{Q})$. Then
 $$\overline{\Eff}^1\bigl(\mathbb{P}_X(E)\bigr) = \Bigl\{x\lambda_E + \pi^*\gamma \mid x\in \mathbb{R}_{\geq 0}, \gamma\in \overline{\Eff}^1(X)\Bigr\},$$
 where $\pi:\mathbb{P}_X(E) \longrightarrow X$ is the projection and $\lambda_E := c_1(\mathcal{O}_{\mathbb{P}_X(E)}(1)) - \frac{1}{r}\pi^*c_1(E)$.
 \end{thm}
We  give several examples to illustrate various applications of Theorem \ref{thm1.1} (see Example \ref{exm4.2}, Example \ref{exm4.4}, Example \ref{exm4.3} etc).

Next we restrict our attention in finding pseudo-effective cones of divisors on $\mathbb{P}_X(E)$ over higher dimensional varieties $X$ where either $E$ is an unstable with respect some polarization or $c_2\bigl(\End(E)\bigr)\neq 0$. In this context,
we  prove the following Theorems.
\begin{thm}\label{thm1.2}
 \rm Let $X$ be a smooth complex irreducible projective variety and let $E=\mathcal{O}_X\oplus \mathcal{L}$ be a rank 2 completely decomposable vector bundle on $X$ with $\mathcal{L}$ being nef.  We denote the numerical class of $\mathcal{O}_{\mathbb{P}(E)}(1)$ by $\xi$. Consider the following 2 cases :
 \begin{enumerate}
  \item  \bf Case (i) : \rm Let $\mathcal{L}\equiv 0$. Then $E$ is slope semistable  with $c_2\bigl(\End(E)\bigr) = 0$. In this case, by Theorem \ref{thm1.1} we have
  $$\overline{\Eff}^1\bigl(\mathbb{P}(E)\bigr) = \Bigl\{x\xi+\pi^*\gamma \mid x\in \mathbb{R}_{\geq 0}, \gamma \in \overline{\Eff}^1(X)\Bigr\}.$$

  \item \bf Case (ii) : \rm  Let $\mathcal{L}$ be not numerically trivial. In this case, $E$ is $H$-unstable for any ample divisor $H$, and
 $$\overline{\Eff}^1\bigl(\mathbb{P}(E)\bigr) = \Bigl\{ x\bigl(\xi-c_1(\pi^*\mathcal{L})\bigr) + \pi^*\gamma  \mid x\in \mathbb{R}_{\geq0}, \gamma \in \overline{\Eff}^1(X)\Bigr\}.$$
 \end{enumerate}
 \end{thm}
\begin{thm}\label{thm1.3}
 Let $X$ be a smooth irreducible complex projective variety of Picard number 1, and $L_X$ be an ample generator of $N^1(X)_{\mathbb{R}}$. Let $E = \mathcal{L}_1\oplus \mathcal{L}_2\oplus \cdots\oplus \mathcal{L}_r$  be a completely decomposable vector bundle of rank $r$ and $l=\max\bigl\{\mathcal{L}_i\cdot L_X^{d-1}\mid 1\leq i\leq r\bigr\}\geq 0$. We denote the numerical class of $\mathcal{O}_{\mathbb{P}(E)}(1)$ by $\xi$. Consider the following two cases :
 \begin{enumerate}
  \rm \item \bf Case (i) \rm: Suppose $\mathcal{L}_1\cdot L_X^{d-1} = \mathcal{L}_2\cdot L_X^{d-1} = \cdots = \mathcal{L}_r\cdot L_X^{d-1}$. Then $E$ is slope $L_X$-semistable with $c_2\bigl(\End(E)\bigr) = 0$, and the pseudo-effective cone $$\overline{\Eff}^1\bigl(\mathbb{P}_X(E)\bigr) = \Bigl\{x_0\bigl(\xi-c_1(\pi^*(E)\bigr)+x_1\pi^*(L_X) \mid x_0,x_1\in \mathbb{R}_{\geq 0}\Bigr\}.$$
 \item \bf Case (ii) \rm: Suppose $E$ is $L_X$-unstable.
 In this case, the pseudo-effective cone $$\overline{\Eff}^1\bigl(\mathbb{P}_X(E)\bigr) = \Bigl\{x_0\bigl(\xi-l\pi^*L_X\bigr)+x_1\pi^*(L_X) \mid x_0,x_1\in \mathbb{R}_{\geq 0}\Bigr\}.$$
\end{enumerate}
\end{thm}

 Let $L$ be a line bundle on  a smooth irreducible projective variety $X$ of dimension $n$. The volume of $L$, denoted by $\vol_X(L)$, is the non-negative number defined as follows:
$$\vol_X(L):=\limsup\limits_{n \to \infty}\frac{\dim_kH^0(X,L^{\otimes m})}{m^n/n!}.$$
The limsup is actually a limit (see \cite{F},\cite{T}). The volume function is  $n$-homogeneous and invariant under numerical equivalence.  Hence we have a well-defined mapping from the N\'{e}ron-Severi group $N^1(X)_{\mathbb{Q}}$ to $\mathbb{Q}_{\geq 0 }$ which we again denote by $\vol_X$. In fact, the function $$\vol_X : N^1(X)_{\mathbb{Q}} \mapsto \mathbb{Q}_{\geq 0 }$$
$$\xi \mapsto \vol_X(\xi)$$ on $N^1(X)_{\mathbb{Q}}$ extends uniquely to a continuous function $\vol_X: N^1(X)_{\mathbb{R}} \longrightarrow \mathbb{R}_{\geq 0}$ on the real N\'{e}ron-Severi group $N^1(X)_{\mathbb{R}}$.  By definition of big divisors, a line bundle $L$ is big if and only if $\vol_X(L) >0$. If $L$ is an ample line bundle, then by asymptotic Riemann-Roch one has $\vol_X(L) =L^n$. In particular, the volume function is polynomial  on the ample cone of $X$.  However, on the larger cone of big divisors, the behaviour of the volume function is not completely understood in general, and in such cases the volume functions are not so easy to calculate. More details about volume functions can be found in \cite{ELMNP} and \cite{L2}.

When $X$ is a surface, then the volume function is piecewise quadratic on the pseudoeffective cone of the surface $X$, and the volume function can be computed using the Zariski decomposition of pseudoeffective divisors on $X$ (see Chapter 2, \cite{L1}).

When $X$ is a toric variety, the pseudo-effective cone $\overline{\Eff}^1(X)$ is a polyhedral cone and it carries a decomposition
into finitely many closed polyhedral subcones on which the volume function vol$_X$ is given by a polynomial of degree equal to $\dim(X)$. The piecewise polynomial nature of vol$_X$ in the toric setting is illustrated in \cite{ELMN03}.

Bauer, Küronya, and Szemberg \cite{BKS03} have used Zariski decompositions to study the volume function when $X$ is a smooth projective surface. They show that in this case Big$(X)$ admits a decomposition into locally polyhedral chambers on which vol$_X(\xi)$ is a quadratic polynomial function of $\xi$. On the other
hand, they also exhibit a threefold $X$ for which vol$_X(\xi)$ is not piecewise polynomial (although it is locally analytic in this example). It remains a very
interesting open question to say something about the nature of the function
\begin{center}
vol$_X : N^1(X)_{\mathbb{R}} \longrightarrow  \mathbb{R}$
\end{center}
in general. The natural expectation here is that there exists a “large” open set $U\subseteq N^1 (X)_{\mathbb{R}}$ such that vol$_X$ is real analytic on each connected component of $U$.

In \cite{C}, the volume function has been calculated on projective bundle $\mathbb{P}_C(E)$ over a smooth irreducible curve $C$ defined over a field of characteristic zero . It is proved in \cite{C} that there is a subdivision of the real N\'{e}ron-Severi group $N^1(\mathbb{P}_C(E))_{\mathbb{R}}$ into sectors such that the volume function is polynomial on each sector.

In his thesis \cite{W}, Wolfe considers the case of split $\mathbb{P}^1$-bundle on a smooth projective variety $X$. More precisely, he computed the volume function on $\mathbb{P}_X(L_1\oplus L_2)$ for some line bundles $L_1$ and $L_2$. It is shown in \cite{W} that when $X$ is a surface $S$, then there exists a dense, open subset $U\subseteq $ Big$\bigl(\mathbb{P}_S(L_1\oplus L_2)\bigr)$ inside the big cone so that $\vol\vert_U$ is real analytic.

In \cite{ELMNP}, the following question is suggested about volume function on projective bundles over higher dimensional varieties (see Problem 2.13, Page 389, \cite{ELMNP}).
\begin{qn} What is the behaviour of volume function of line bundles on projective bundles $\mathbb{P}_X(E)$ over  higher dimensional varieties $X$? It is also interesting to see to what extent the geometry of $E$ is related to the volume function on $\mathbb{P}_X(E)$.
\end{qn}
Motivated by the above question, we calculate the volume of certain line bundles $L$ on the fibre product $\mathbb{P}_C(E)\times_C \mathbb{P}_C(F)$ of two projective bundles over a  smooth complex projective irreducible curve  $C$, where $E$ and $F$ are non-zero vector bundles on $C$  of rank $e$ and $f$ respectively. Our proofs closely follow the ideas in \cite{C}. Note that the fibre product $X=\mathbb{P}_C(E)\times_C \mathbb{P}_C(F)$ has Picard number 3, and the behaviours of different cones inside $N^1(X)_{\mathbb{R}}$ are hard to predict.

 There exists a unique filtration, called the Harder-Narasimhan filtration of $E$, as follows
\begin{align}\label{s0}
0=E_d\subsetneq E_{d-1}\subsetneq \cdots \subsetneq E_1\subsetneq E_0 =E
\end{align}
such that each sub-quotient $E_i/E_{i+1}$ is semistable for each $i\in\{0,1,2,\cdots,{d-1}\}$, and $$\mu_0<\mu_1<\cdots<\mu_{d-1}$$ where $$\mu_i = \mu(E_i/E_{i+1}) = \frac{\deg(E_i/E_{i+1})}{\rk(E_{i}/E_{i+1})}$$ for each $i$. Then one can define a probability measure $\nu_E$ on $\mathbb{R}$ as follows:
$$\nu_E := \frac{1}{\rk(E)}\sum\limits_{i=0}^{d-1}\{\rk(E_{i})-\rk(E_{i+1})\}\delta_{\mu_i},$$
where $\delta_x$ is the Dirac measure concentrated on $\{x\}$. For any real number $\varepsilon > 0$, let $T_{\varepsilon}$ be the operator on the set of all Borel probability measures on $\mathbb{R}$ which is defined by the following property: if $\nu$ is Borel probability measure and if $f$ is a continuous function of compact support on $\mathbb{R}$, then
$$\int\limits_{\mathbb{R}} f(x)T_{\varepsilon}\nu (dx) = \int\limits_{\mathbb{R}} f(\varepsilon x)\nu(dx).$$
Let 
\begin{align}\label{seq2}
 0=F_h\subsetneq F_{h-1}\subsetneq F_{h-2}\subsetneq\cdots\subsetneq F_{1} \subsetneq F_0 = F
\end{align}
be the Harder-Narasimhan filtration of $F$, and let $\mu'_j = \mu(F_j/F_{j+1})$ for each $j$ so that we have $$\mu'_{0} < \mu'_{1} < \cdots <\mu'_{h-2} < \mu'_{h-1}.$$
Consider the following commutative fibre product diagram:
\begin{center}
 \begin{tikzcd} 
\mathbb{P}(\mathcal{E}) = \mathbb{P}_C(F)\times_C \mathbb{P}_C(E) = \mathbb{P}(\psi^*E) \arrow[r, "\pi_2"] \arrow[d, "\pi_1"]
& \mathbb{P}_C(E)  \arrow[d,""]\\
 \mathbb{P}_C(F) \arrow[r, "\psi" ]
& C
\end{tikzcd}
\end{center}
We call  $\pi = \psi \circ \pi_1$ and $\mathcal{E} = \psi^*E$. We denote by $K:=k(C)$ the function field of the smooth curve $C$. Consider a line bundle $L = \mathcal{O}_{\mathbb{P}(\mathcal{E})}(m) \otimes \pi_1^*\mathcal{O}_{\mathbb{P}(F)}(l)\otimes \pi_1^*\psi^*\mathcal{M}$ on $\mathbb{P}(\mathcal{E})$, for some line bundle $\mathcal{M}$ of degree $a$ on $C$ and for some integers $m,l>0$.

We consider the vector $\underline{s} := (s_1,s_2,\cdots,s_e,s'_1,s'_2,\cdots,s'_f) \in \mathbb{R}^{e}\times\mathbb{R}^{f}$ such that the value $\mu_i$ appear exactly $r_i = \rk(E_i/E_{i+1})$ times in the first $e$ coordinates and the value $\mu_j'$ appears exactly $r'_j = \rk(F_j/F_{j+1})$ times in the last $f$ coordinates.  Consider the simplices $\Delta_m$ and $\Delta_l$  defined as follows:
$$\Delta_m = \Bigl\{(x_1,x_2,\cdots,x_e)\in\mathbb{R}^e \mid 0\leq x_i \leq m, \sum\limits_{i=1}^{e}x_i = m\Bigr\},$$
$$\Delta_l=\Bigl\{(x'_1,x'_2,\cdots,x'_f)\in \mathbb{R}^f\mid 0\leq x'_j \leq l,\sum\limits_{i=1}^fx'_j= l\Bigr\}.$$
We prove the following useful result to compute the volume of the line bundle $L$ on the fiber product $X = \mathbb{P}_C(F)\times_C \mathbb{P}(E)$.
\begin{thm}
 Let $\eta$ be the Lebesgue measure on $\Delta_m\times \Delta_l$ normalized such that $\eta(\Delta_m\times \Delta_l) = m^e\cdot n^f$. Then  $T_{\frac{1}{n}}\nu_{\Sym^{mn}(E)\otimes \Sym^{ln}(F)}$ converges vaguely (see Definition \ref{defi5.1}) to $\phi_{\underline{s}*}\eta$, where $$\phi_{\underline{s}} : \Delta_m \times \Delta_l\longrightarrow \mathbb{R}$$ such that\hspace{3cm} $(x_1,x_2,\cdots,x_e,x'_1,\cdots,x'_f)\mapsto\sum\limits_{i=1}^{e} s_ix_i+\sum\limits_{j=1}^{f}s'_jx'_j$.
 \end{thm}
As a corollary, we compute the volume of  certain line bundles $L$ on the fiber product $X = \mathbb{P}_C(F)\times_C \mathbb{P}(E)$.
\begin{corl}
 Let $L = \mathcal{O}_{\mathbb{P}(\mathcal{E})}(m) \otimes \pi_1^*\mathcal{O}_{\mathbb{P}(F)}(l)\otimes \pi_1^*\psi^*\mathcal{M}$ be a line bundle on $X= \mathbb{P}_C(F)\times_C \mathbb{P}_C(E)$ for some line bundle $\mathcal{M}$ of degree $a$ on a smooth irreducible complex curve $C$ such that $L_{K}$ is big. Then
 $$\vol_X(L) =\dim(X)\vol_{X_K}(L_K) \int (x+a)_+ \phi_{\underline{s}*}\eta (dx),$$
 where $x_+ = \max\{x,0\}$.
 \end{corl}
 In particular, we show that the volume function $\Phi(m,l,a):=\vol_X(L)$ of such line bundles $L$ on the fiber product $X=\mathbb{P}_C(F)\times_C\mathbb{P}_C(E)$ of two ruled surfaces is of polynomial type in the variable $m,l,a$.
 \section{Notation and Convention}
Throughout this article, all the algebraic varieties are assumed to be irreducible and defined over the field of characteristic zero.

The projective bundle $\mathbb{P}_X(E)$ associated to a vector bundle 
$E$ over a projective variety $X$ is defined as $\mathbb{P}_X(E) :=  \Proj\bigl(\bigoplus\limits_{m\geq  0} \Sym^m(E)\bigr)$ together with the projection map $\pi : \mathbb{P}_X(E)\longrightarrow X$.
We will simply write $\mathbb{P}(E)$ whenever the base space $X$ is clear from the context. The numerical class of the tautological line bundle $\mathcal{O}_{\mathbb{P}(E)}(1)$ will be denoted by $\xi$ unless otherwise specified. The numerical class of a divisor $D$ will be denoted by  $\bigl[D\bigr]$.
The rank of a vector bundle $E$ will be denoted by $\rk(E)$.
The set of non-negative  real numbers (resp. rational numbers) will be denoted by $\mathbb{R}_{\geq 0}$ (resp. $\mathbb{Q}_{\geq 0}$).
 \section{Preliminaries}
 In this section we recall the definition of pseudo-effective cones and its main properties. We also recall the characterization of semistable bundles with vanishing discriminant over higher dimensional projective varieties which we will use to prove one of our main observations.
 \subsection{Nef cone and Pseudo-effective cone}
 Let $X$ be a smooth projective variety of dimension $n$ and $\mathcal{Z}_1(X)$ (respectively $\mathcal{Z}^1(X) = \Div(X)$) denotes the free abelian group generated by $1$-dimensional (respectively $1$-codimensional) subvarieties on $X$.
 Two cycles $Z_1$,$Z_2 \in \mathcal{Z}_1(X)$ are said to be numerically equivalent, denoted by $Z_1\equiv Z_2$ if $Z_1\cdot \gamma =  Z_2\cdot\gamma $
  for all $\gamma \in \mathcal{Z}^1(X)$. The \it numerical groups \rm  $ N_1(X)_{\mathbb{R}}$ are defined as the quotient of
  $\mathcal{Z}_1(X) \otimes \mathbb{R}$ modulo numerical equivalence.
  
  Let 
  \begin{center}
  $\Div^0(X) := \bigl\{ D \in \Div(X) \mid D\cdot C = 0 $ for all curves $C$ in $X \bigr\} \subseteq \Div(X)$.
  \end{center}
 be the subgroup of $\Div(X)$ consisting of numerically trivial divisors. The quotient $\Div(X)/\Div^0(X)$ is called the N\'{e}ron Severi group of $X$, and is denoted by $N^1(X)_{\mathbb{Z}}$.
   The N\'{e}ron Severi group  $N^1(X)_{\mathbb{Z}}$ is a free abelian group of finite rank.
 Its rank, denoted by $\rho(X)$ is called the Picard number of $X$. In particular, $N^1(X)_{\mathbb{R}}$ is called the real N\'{e}ron 
 Severi group and $N^1(X)_{\mathbb{R}}  := N^1(X)_{\mathbb{Z}} \otimes \mathbb{R} := \bigl(\Div(X)/\Div^0(X)\bigr) \otimes \mathbb{R}$.
The intersection product induces a perfect pairing
 \begin{align*}
N^1(X)_{\mathbb{R}} \times N_1(X)_{\mathbb{R}} \longrightarrow \mathbb{R}
\end{align*}
which implies $N^1(X)_{\mathbb{R}} \cong (N_1(X)_{\mathbb{R}})^\vee$. The convex cone generated by the set of all effective $1$-cycles in $N_1(X)_\mathbb{R}$ is denoted by $\Eff_1(X)$ and its closure $\overline{\Eff}_1(X)$ is called the \it pseudo-effective cone \rm  of $1$-cycles in $X$.

Similarly, the convex cone generated by the set of all effective divisors in $N^1(X)_\mathbb{R}$ is denoted by $\Eff^1(X)$ and its closure $\overline{\Eff}^1(X)$ is called the \it pseudo-effective cone \rm  of divisors in $X$. The \it nef cone \rm of divisors are defined as follows :
\begin{align*}
 \Nef^1(X) := \bigl\{ \alpha \in N^1(X) \mid \alpha \cdot \beta \geq 0 \hspace{2mm} \forall \beta \in \overline{\Eff}_1(X)\bigr\}.  
\end{align*}
The interior int$\bigl(\overline{\Eff}^1(X)\bigr)$ is the Big cone of big divisors in $N^1(X)_{\mathbb{R}}$,  denoted by Big$(X)$.

An irreducible curve $C$ in $X$ is called \it movable \rm if there exists an algebraic family of irreducible curves $\{C_t\}_{t\in T}$ such 
that $C = C_{t_0}$ for some $t_0 \in T$ and $\bigcup_{t \in T} C_t \subset X$ is dense in $X$.

A class $\gamma \in N_1(X)_{\mathbb{R}}$ is called movable if there exists a movable curve $C$
such that $\gamma = [C]$ in $N_1(X)_{\mathbb{R}}$. The closure of the cone generated by movable classes in
$N_1(X)_{\mathbb{R}}$, denoted by $\overline{\Mov}_1(X)$ is called the \it movable cone\rm. By \cite{BDPP13} $\overline{\Mov}_1(X)$ is
the dual cone to $\overline{\Eff}^1(X)$. We have the following inclusions among these cones : $\Nef^1(X) \subseteq \overline{\Eff}^1(X)$ and $\overline{\Mov}_1(X) \subseteq \overline{\Eff}_1(X)$ always. We refer the
reader to \cite{L1},\cite{L2} for more details about these cones.
\subsection{Semistability of Vector bundles}
Let $X$ be a smooth complex projective variety of dimension $n$ with a fixed ample line bundle $H$ on it.
For a torsion-free coherent sheaf $\mathcal{G}$ of rank $r$ on $X$, the $H$-slope of $\mathcal{G}$ is defined as 
\begin{align*}
\mu_H(\mathcal{G}) := \frac{c_1(\mathcal{G})\cdot H^{n-1}}{r} \in \mathbb{Q}.
\end{align*}
A non-zero torsion-free coherent sheaf $\mathcal{G}$ on $X$ is said to be $H$-semistable if $\mu_H(\mathcal{F}) \leq \mu_H(\mathcal{G})$ for
all subsheaves $\mathcal{F}$ of $\mathcal{G}$.
A vector bundle $E$ on $X$ is called $H$-unstable if it is not $H$-semistable. For every vector bundle $E$ on $X$, there is a unique filtration
\begin{align*}
 0 = E_d \subsetneq E_{d-1} \subsetneq E_{d-2} \subsetneq\cdots\subsetneq E_{1} \subsetneq E_0 = E
\end{align*}
of subsheaves of $E$, called the Harder-Narasimhan filtration of $E$, such that $E_i/E_{i+1}$ is $H$-semistable torsion-free sheaf for each $i \in \{ 0,1,2,\cdots,d-1\}$
and $$\mu_H\bigl(E_{d-1}/E_{d}\bigr) > \mu_H\bigl(E_{d-2}/E_{d-1}\bigr) >\cdots> \mu_H\bigl(E_{0}/E_{1}\bigr).$$
We define $Q_1:= E_{0}/E_{1}$ and $\mu_{\min}(E) := \mu_H(Q_1) = \mu_H\bigl(E_{0}/E_{1}\bigr)$ and $\mu_{\max}(E) := \mu_H\bigl(E_{d-1}/E_d\bigr).$

For a vector bundle $E$ of rank $r$ over $X$, the element $c_2\bigl(\End(E)\bigr) \in H^4(X,\mathbb{Q})$ is called the discriminant of $E$. We recall the following result about semistable bundle with vanishing discriminant.
\begin{thm}\label{thm3.1}\rm[\cite{N99}  or  [Theorem 1.2,\cite{B-B}]
\it  Let $E$ be a vector bundle of rank $r$ on a smooth complex projective variety $X$, and $\pi : \mathbb{P}(E) \longrightarrow X$ be the projection. Then the following are equivalent
 
 \rm(1) \it $E$ is semistable and $c_2\bigl(\End(E)\bigr) = 0$.
 
 \rm(2) \it $\lambda_E := c_1(\mathcal{O}_{\mathbb{P}(E)}(1)) - \frac{1}{r}\pi^*c_1(E) \in \Nef^1\bigl(\mathbb{P}(E)\bigr)$.
 
 \rm(3) \it For every pair of the form $(\phi,C)$, where $C$ is a smooth projective curve and $\phi : C \longrightarrow X$ is a non-constant morphism, $\phi^*(E)$ is semistable.
\end{thm}
\rm Since nefness of a line bundle does not depend on the fixed ample line bundle $H$ on $X$, the Theorem \ref{thm3.1} implies that the semistability of a vector bundle $E$ with $c_2\bigl(\End(E)\bigr) = 0$ is independent of the fixed ample bundle $H$. We have the following lemma as easy application of Theorem \ref{thm3.1}.
\begin{lem}\label{lem2.2}
 Let $\psi : X \longrightarrow Y$ be a morphism between two smooth complex projective varieties and $E$ is a semistable bundle on $Y$ with $c_2\bigl(\End(E)\bigr) = 0$. Then the pullback bundle $\psi^*(E)$ is also semistable with $c_2\bigl(\End(\psi^*E)\bigr) = 0$, and for any positive integer $m$ and any line bundle $\mathcal{L}$ on $X$, $\Sym^m(E)\otimes \mathcal{L}$ is also semistable bundle on $X$ with $c_2\bigl(\End(\Sym^m(E)\otimes \mathcal{L})\bigr) = 0$.
\end{lem}
\section{pseudoeffective cone of projective bundles}
\begin{thm}\label{thm4.1}
 Let $E$ be a semistable vector bundle of rank $r$ on a smooth complex irreducible projective  variety $X$ such that $c_2\bigl(\End(E)\bigr) = 2rc_2(E)-(r-1)c^2_1(E) = 0 \in H^4(X,\mathbb{Q})$. Then 
 $$\overline{\Eff}^1\bigl(\mathbb{P}(E)\bigr) = \Bigl\{x\lambda_E + \pi^*\gamma \mid x\in \mathbb{R}_{\geq 0}, \gamma\in \overline{\Eff}^1(X)\Bigr\}.$$
 
 In particular, if $\overline{\Eff}^1(X)$ is a finite polyhedron generated by $\{E_1,E_2,\cdots,E_n\}$,
 then $\overline{\Eff}^1\bigl(\mathbb{P}(E)\bigr)$ is also a finite polyhedron and $$\overline{\Eff}^1\bigl(\mathbb{P}(E)\bigr) = \Bigl\{x_0\lambda_E + \sum\limits_{i=1}^nx_i\pi^*(E_i) \mid x_0,x_1,\cdots,x_n \in \mathbb{R}_{\geq 0}\Bigr\}.$$
\end{thm}
\begin{proof}
Let $D$ be an effective divisor on $\mathbb{P}(E)$ such that 
$\mathcal{O}_{\mathbb{P}(E)}(D) \simeq \mathcal{O}_{\mathbb{P}(E)}(m) \otimes \pi^*(\mathcal{L})$ for some integer $m$ and a line bundle $\mathcal{L} \in \Pic(X)$. Then 
 \begin{align*}
  H^0\bigl(\mathbb{P}(E), \mathcal{O}_{\mathbb{P}(E)}(m) \otimes \pi^*\mathcal{L}\bigr) = H^0\bigl(X, \Sym^m(E) \otimes \mathcal{L}\bigr) \neq 0 ,
 \end{align*}
 which implies $m \geq 0$. Let $\gamma = [C] $ be a movable class in $N_1(X)_{\mathbb{R}}$. Then $C$ belongs to an algebraic family of curves
$\{C_t\}_{t\in T}$ such that  $\bigcup_{t \in T} C_t$ covers a dense subset of $X$. So we can find a curve
$C_{t_1}$ in this family such that  
\begin{align*}
 H^0\bigl( C_{t_1}, \Sym^m(E)\vert_{C_{t_1}} \otimes \mathcal{L}\vert_{C_{t_1}}\bigr) \neq 0.
\end{align*}
Let $\eta_{t_1} :\tilde{C_{t_1}} \longrightarrow C_{t_1}$ be the normalization of the curve $C_{t_1}$ and we call $\phi_{t_1} :=  i\circ \eta_{t_1} $ where $i : C_{t_1} \hookrightarrow X$ is the inclusion.
As $E$ is a semistable bundle on $X$ with $c_2\bigl(\End(E)\bigr) = 0$, using Theorem \ref{thm3.1} and Lemma \ref{lem2.2} we have  $\phi_{t_1}^*\bigl(\Sym^m(E) \otimes \mathcal{L}\bigr)$ is also semistable on $\tilde{C_{t_1}}$. Since $\eta_{t_1}$ is a surjective map, we have
\begin{align*}
 H^0\bigl(\eta_{t_1}^*(\Sym^m(E)\vert_{C_{t_1}} \otimes \mathcal{L}\vert_{C_{t_1}})\bigr)  = H^0\bigl(\phi_{t_1}^*(\Sym^m(E) \otimes \mathcal{L})\bigr)\neq 0.
\end{align*}
Then we have an injection $0\longrightarrow \mathcal{O}_{\tilde{C_{t_1}}} \longrightarrow  \phi_{t_1}^*(\Sym^m(E) \otimes \mathcal{L})$.
Recall that $\phi_{t_1}^*(\Sym^m(E) \otimes \mathcal{L})$ is semistable.
This implies that  $c_1\bigl(\phi_{t_1}^*\bigl(\Sym^m(E) \otimes \mathcal{L}\bigr)\bigr) \geq 0$ , and hence
$$c_1\bigl(\Sym^m(E) \otimes \mathcal{L}\bigr)\cdot C_{t_1} = \bigl\{ c_1\bigl(\Sym^m(E) \otimes \mathcal{L}\bigr) \cdot \gamma \bigr\} \geq 0$$
for a movable class $\gamma \in N_1(X)_{\mathbb{R}}$. Using the duality property of movable cone $\overline{\Mov}_1(X)$ we conclude that
\begin{center}
$ c_1\bigl(\Sym^m(E) \otimes \mathcal{L}\bigr) = c_1\bigl(\pi_*\mathcal{O}_{\mathbb{P}(E)}(D)\bigr) \in \overline{\Eff}^1(X)$.
\end{center}
Now 
 $c_1\bigl(\pi_*\mathcal{O}_{\mathbb{P}(E)}(D)\bigr)$
 = $c_1\bigl(\Sym^m(E)\otimes \mathcal{L}\bigr) = \rk\bigl(\Sym^m(E)\bigr)\bigl\{ \frac{m}{r}c_1(E) + c_1(\mathcal{L}) \bigr\} \in \overline{\Eff}^1(X)$, so that $$\frac{m}{r}c_1(E) + c_1(\mathcal{L}) \in \overline{\Eff}^1(X).$$
Now
$\mathcal{O}_{\mathbb{P}(E)}(m) \otimes \pi^*(\mathcal{L})$
$\equiv m\bigl\{ c_1(\mathcal{O}_{\mathbb{P}(E)}(1)) - \frac{1}{r}\pi^*c_1(E)\bigr\} + \frac{m}{r}\pi^*c_1(E) + \pi^*c_1(\mathcal{L})$

$\equiv m\lambda_E + \pi^*\bigl(\frac{m}{r}c_1(E) + c_1(\mathcal{L})\bigr) \in \overline{\Eff}^1\bigl(\mathbb{P}(E)\bigr)$, as $\lambda_E  \in \Nef^1\bigl(\mathbb{P}(E)\bigr) \subseteq \overline{\Eff}^1\bigl(\mathbb{P}(E)\bigr).$

This shows that 
$$ \overline{\Eff}^1\bigl(\mathbb{P}(E)\bigr) = \Bigl\{x\lambda_E + \pi^*\gamma \mid x\in \mathbb{R}_{\geq 0}, \gamma\in \overline{\Eff}^1(X)\Bigr\}. $$
\end{proof}

Next we illustrate Theorem \ref{thm4.1} through various examples.
\begin{exm}\label{exm4.2}
\rm Let $X$ be a smooth complex projective variety of dimension $n$ with Picard number $\rho(X) = 1$. We fix an ample generator $L_X$ of $N^1(X)_{\mathbb{Z}}\cong \mathbb{Z}$. Let $E$ be a semistable vector bundle of rank $r$ on $X$ with $c_2\bigl(\End(E)\bigr) = 0$. Note that $\overline{\Eff}^1(X)$ is generated by $L_X$, and hence  from Theorem \ref{thm4.1} we conclude
\begin{center}
$\overline{\Eff}^1\bigl(\mathbb{P}(E)\bigr) = \Bigl\{ x_0\lambda_E + x_1(\pi^*L_X) \mid x_0,x_1 \in \mathbb{R}_{\geq 0}\Bigr\}$.
\end{center}
For example, let $$E = \mathcal{L}_1\oplus \mathcal{L}_2 \oplus\cdots\oplus \mathcal{L}_r$$ be a direct sum of line bundles on $X$ such that $\mu_{L_X}(\mathcal{L}_1) = \mu_{L_X}(\mathcal{L}_2) = \cdots = \mu_{L_X}(\mathcal{L}_r)$.  Then $E$ is semistable with $c_2\bigl(\End(E)\bigr) =  0$. Therefore, $$\overline{\Eff}^1\bigl(\mathbb{P}(\mathcal{L}_1\oplus \mathcal{L}_2\oplus\cdots\oplus \mathcal{L}_r)\bigr)= \Bigl\{x_0\lambda_E+x_1(\pi^*L_X)\mid x_0,x_1\in \mathbb{R}_{\geq 0}\Bigr\}.$$
\end{exm}
We now discuss the pseudo-effective cones $\overline{\Eff}^1\bigl(\mathbb{P}_X(E)\bigr)$ of some projective bundles $\mathbb{P}_X(E)$ where the base variety $X$ has Picard number $\rho(X) \geq 2$ or $E$ is not semistable. Note that in these cases $\rho\bigl(\mathbb{P}_X(E)\bigr) =\rho(X)+1 \geq 3$ and hence these cones in three dimensional real vector spaces may not be finite polyhedra type in general. We give an example of a projective bundle $\mathbb{P}_X(E)$ having non-polyhedral pseudo-effective cone.
\begin{exm}\label{exm4.4}
 \rm Let $C$ be a general elliptic curve and $X = C\times C$ be the self product. Then $X$ is an abelian surface and $\overline{\Eff}^1(X)$ is a non-polyhedral cone (see Lemma 1.5.4 in \cite{L1}). Let $p_i : X \longrightarrow C$ be the projection maps. For any semistable vector bundle $F$ of rank $r$ on $C$, the pullback bundle $E:= p_i^*(F)$ is a semistable bundle with $c_2\bigl(\End(E)\bigr) = 0$. Therefore by Theorem \ref{thm4.1} $$\overline{\Eff}^1\bigl(\mathbb{P}(E)\bigr) = \Bigl\{a\lambda_E+\pi^*\gamma \mid \gamma\in \overline{\Eff}^1(X), a\in \mathbb{R}_{\geq 0}\Bigr\}.$$
 In this case $\overline{\Eff}^1\bigl(\mathbb{P}(E)\bigr)$ is a non-polyhedral cone.
\end{exm}
\begin{exm}\label{exm4.3}
 \rm Let $E_1,E_2,E_3,\cdots,E_l$ be finitely many vector bundles on a smooth complex projective curve $C$. Then the pseudo-effective cone $\overline{\Eff}^1\bigl(\mathbb{P}_C(E_1)\times_C \mathbb{P}_C(E_2)\times_C \cdots \times_C \mathbb{P}_C(E_l)\bigr)$ of the fibre product is computed in \cite{M21}.
 \end{exm}
Next we restrict our attention in finding pseudo-effective cone of $\mathbb{P}_X(E)$ over higher dimensional varieties $X$ where $c_2\bigl(\End(E)\bigr) \neq 0$. Our first result in this direction is the following.

 \begin{thm}\label{thm4.5}
 \rm Let $X$ be a smooth complex irreducible projective variety of dimension $d$ and let $E=\mathcal{O}_X\oplus \mathcal{L}$ be a rank 2 completely decomposable vector bundle on $X$ with $\mathcal{L}$ being nef.  We denote the numerical class of $\mathcal{O}_{\mathbb{P}(E)}(1)$ by $\xi$. Consider the following 2 cases :
 \begin{enumerate}
  \item  \bf Case (i) : \rm Let $\mathcal{L}\equiv 0$. Then $E$ is slope semistable with  $c_2\bigl(\End(E)\bigr) = 0$. In this case, by Theorem \ref{thm1.1} we have
  $$\overline{\Eff}^1\bigl(\mathbb{P}(E)\bigr) = \Bigl\{x\xi+\pi^*\gamma \mid x\in \mathbb{R}_{\geq 0}, \gamma \in \overline{\Eff}^1(X)\Bigr\}.$$
\item \bf Case (ii) : \rm  Let $\mathcal{L}$ be not numerically trivial. In this case, $E$ is $H$-unstable for any ample divisor $H$, and
 $$\overline{\Eff}^1\bigl(\mathbb{P}(E)\bigr) = \Bigl\{ x\bigl(\xi-c_1(\pi^*\mathcal{L})\bigr) + \pi^*\gamma  \mid x\in \mathbb{R}_{\geq0}, \gamma \in \overline{\Eff}^1(X)\Bigr\}.$$
 \end{enumerate}
 \end{thm}
 \begin{proof}
 Note that as $\mathcal{L}$ is nef, we have $\mathcal{L}\cdot H^{d-1} \geq 0$ for any ample class $H$ in $N^1(X)_{\mathbb{R}}$.
 We consider the following two cases :
\begin{itemize}
 \item {\bf Case (i) :} \rm  Let $\mathcal{L}\equiv 0$.
 Then $E$ is slope semistable with $c_2\bigl(\End(E)\bigr) = 0$. In this case, $\overline{\Eff}^1\bigl(\mathbb{P}(E)\bigr)$ can be computed using Theorem \ref{thm4.1}, and that is
$$\overline{\Eff}^1\bigl(\mathbb{P}(E)\bigr) =  \Bigl\{x\xi+\pi^*\gamma \mid \gamma \in \overline{\Eff}^1(X), x \in \mathbb{R}_{\geq 0} \Bigr\}.$$
\item {\bf Case (ii) :} \rm Let $\mathcal{L}$ be not numerically trivial and $\mathcal{L}\cdot H^{d-1} > 0$ for any ample class $H$ in $N^1(X)_{\mathbb{R}}$ (and hence for every ample class $H$ in $N^1(X)_{\mathbb{R}}$). In this case, $E$ is $H$-unstable for any ample class $H\in N^1(X)_{\mathbb{R}}$,
and the Harder Narasimhan filtration of $E$ with respect to ample class $H$ is given by
  $$0  \subsetneq \mathcal{L}  \subsetneq E .$$
Let us fix notation for the quotient $Q_1 := E/\mathcal{L} = \mathcal{O}_{X}$.

   Let $a\xi+\pi^*\gamma  \in   \overline{\Eff}^1\bigl(\mathbb{P}(E)\bigr)$ for some $\gamma \in N^1(X)_{\mathbb{R}}$ and for some $a \in \mathbb{R}$. Therefore, for any ample class $h$ in $X$, we get
   $$\bigl(a\xi+\pi^*\gamma\bigr)\underbrace{\cdot (\pi^*h)\cdot(\pi^*h)\cdots(\pi^*h)}_{d-times} = a h^d\geq 0.$$
This implies $a \geq 0$ as $h^d > 0$.
Note that $[\mathbb{P}(Q_1)] = \xi - c_1(\pi^*\mathcal{L}).$

   Now
   \begin{align}\label{seq20}
   a\xi+\pi^*\gamma  = a\bigl(\xi-c_1(\pi^*\mathcal{L})\bigr) + \pi^*\bigl(\gamma+ac_1(\mathcal{L})\bigr).
   \end{align}

   Our claim is that $\gamma+ac_1(\mathcal{L}) \in \overline{\Eff}^1(X).$ For any movable curve class $\alpha \in \overline{\Mov}_1(X)$, we note that $\xi\cdot \pi^*\alpha$ is a movable curve class. This is equivalent to the fact that $D\cdot \xi \cdot \pi^*\alpha \geq 0$ for any irreducible divisor $D$ in $\mathbb{P}(E)$. If $D$ does not dominate $X$, then it is a pullback $D=\pi^*D'$ for some irreducible divisor in $X$. In this case the product is $D\cdot \xi \cdot \pi^*\alpha = \xi\cdot \pi^*(D'\cdot \alpha)$ which is nonnegative, the degree of $D'\cdot \alpha$.
   Now assume that $D$ dominates $X$. Let
   $\sigma : \overline{D}\longrightarrow D$ be a resolution of $D$. Then $\mu=\pi\vert_D\circ \sigma : \overline{D}\longrightarrow X$ is a genericallly finite dominant morphism. We want to show that $\sigma^*\xi\vert_D\cdot \mu^*\alpha \geq 0$. Note that $\mu_*(\sigma^*\xi\vert_D)$ is pseudoeffective in this case. As $\alpha\in \overline{\Mov}_1(X)$, we have by projection formula and by applying \cite{BDPP13}, we conclude  $\sigma^*\xi\vert_D\cdot \mu^*\alpha \geq 0$.
   Thus we have $$(a\xi+\pi^*\gamma) \cdot \xi \cdot \pi^*\alpha \geq 0$$

   This implies  $(\gamma+ac_1(\mathcal{L}))\cdot \alpha \geq 0$.
   Hence by duality property of $\overline{\Eff}^1(X)$ and $\overline{\Mov}_1(X)$, we conclude that $$\gamma+ac_1(\mathcal{L})  \in  \overline{\Eff}^1(X).$$

   Thus from (\ref{seq20}), we get

    $$\overline{\Eff}^1\bigl(\mathbb{P}(E)\bigr) = \Bigl\{ a\bigl(\xi-\pi^*c_1(\mathcal{L})\bigr) + \pi^*\gamma  \mid a\in \mathbb{R}_{\geq0}, \gamma \in \overline{\Eff}^1(X)\Bigr\}.$$
   \end{itemize}
  \end{proof}

\begin{xrem}
 \rm The nefness condition of the line bundle $\mathcal{L}$ in Theorem \ref{thm4.5} is a technical one so that the Harder-Narasimhan filtration of $E$ does not depend on the ample divisor $H$ in Case (ii) in Theorem \ref{thm4.5}. The pseudo-effective cone  $\overline{\Eff}^1\bigl(\mathbb{P}(E)\bigr)$ can still be computed without the assumption of nefness of $\mathcal{L}$. However, the computations in those cases will be complicated in general.
\end{xrem}

\begin{thm}\label{thm4.6}
 Let $X$ be a smooth irreducible complex projective variety of Picard number 1, and $L_X$ be an ample generator of $N^1(X)_{\mathbb{R}}$. Let $E = \mathcal{L}_1\oplus \mathcal{L}_2\oplus \cdots\oplus \mathcal{L}_r$  be a completely decomposable vector bundle of rank $r$ and $l=\max\bigl\{\mathcal{L}_i\cdot L_X^{d-1}\mid 1\leq i\leq r\bigr\}\geq 0$. We denote the numerical class of $\mathcal{O}_{\mathbb{P}(E)}(1)$ by $\xi$. Consider the following two cases :
 \begin{enumerate}
  \item \bf Case (i) \rm: Suppose $\mathcal{L}_1\cdot L_X^{d-1} = \mathcal{L}_2\cdot L_X^{d-1} = \cdots = \mathcal{L}_r\cdot L_X^{d-1}$. Then $E$ is slope $L_X$-semistable with $c_2\bigl(\End(E)\bigr) = 0$, and the pseudo-effective cone $$\overline{\Eff}^1\bigl(\mathbb{P}_X(E)\bigr) = \Bigl\{x_0\bigl(\xi-c_1(\pi^*(E)\bigr)+x_1\pi^*(L_X) \mid x_0,x_1\in \mathbb{R}_{\geq 0}\Bigr\}.$$
 \item \bf Case (ii) \rm: Suppose $E$ is $L_X$-unstable.
 In this case, the pseudo-effective cone $$\overline{\Eff}^1\bigl(\mathbb{P}_X(E)\bigr) = \Bigl\{x_0\bigl(\xi-l\pi^*L_X\bigr)+x_1\pi^*(L_X) \mid x_0,x_1\in \mathbb{R}_{\geq 0}\Bigr\}.$$
\end{enumerate}
\begin{proof}
\it Proof of Case (i) : \rm In this case, $E$ is $L_X$-semistable with $c_2\bigl(\End(E)\bigr) = 0$. Therefore,  using Theorem \ref{thm4.1}, we conclude  $$\overline{\Eff}^1\bigl(\mathbb{P}_X(E)\bigr) = \Bigl\{x_0\bigl(\xi-c_1(\pi^*(E))\bigr)+x_1\pi^*(L_X) \mid x_0,x_1\in \mathbb{R}_{\geq 0}\Bigr\}.$$

 \it Proof of Case (ii) : \rm In this case, $E$ is $L_X$-unstable and let
 \begin{align*}
 0 = E_d \subsetneq E_{d-1} \subsetneq E_{d-2} \subsetneq\cdots\subsetneq E_{1} \subsetneq E_0 = E
\end{align*}
be the Harder-Narasimhan filtration of $E$  with respect to $L_X$.

 We have a rational projection map $$p : \mathbb{P}_X(E)  \dashrightarrow \mathbb{P}_X(E_1)$$ defined outside $\mathbb{P}(Q_1)$. The indeterminacies of this rational map are resolved by blowing-up $\mathbb{P}_X(Q_1)$. By following the ideas of Proposition 2.4 in \cite{Fu} (See Remark 2.5 \cite{Fu}), one can find
a locally free sheaf $\mathcal{F}$ on $Y$ such that $\bigl(\Bl_{\mathbb{P}(\mathcal{Q}_1)}\mathbb{P}(E),Y,\eta\bigr)$ is the projective bundle $\mathbb{P}_Y(\mathcal{F})$ over $Y$ with $\eta$ as the projection map.
Thus we have the following commutative diagram : 
\begin{center}
 \begin{tikzcd} 
\mathbb{P}_Y(\mathcal{F}) = \Bl_{\mathbb{P}(\mathcal{Q}_1)}\mathbb{P}(E) \arrow[r, "\eta"] \arrow[d, "\beta"]
& \mathbb{P}(E_1) = Y \arrow[d,"\rho"]\\
\mathbb{P}(E) \arrow[r, "\pi" ]
& X
\end{tikzcd}
\end{center}
where $\mathcal{F}$ sits in the following exact sequence :
\begin{align}\label{seq1}
0 \longrightarrow \mathcal{O}_{\mathbb{P}(E_1)}(1)\longrightarrow \mathcal{F}\longrightarrow \rho^*(Q_1) \longrightarrow 0.
\end{align}
Moreover, if $\gamma := \mathcal{O}_{\mathbb{P}_Y(\mathcal{F})}(1)$, 
$\xi := \mathcal{O}_{\mathbb{P}(E)}(1)$ and $\xi_1 := \mathcal{O}_{\mathbb{P}(E_1)}(1)$,
then $$\gamma = \beta^*\xi, \hspace{2mm} \eta^*\xi_1 = \beta^*\xi - \tilde{E}$$ where $\tilde{E}$ is numerical class of the exceptional divisor of the map $\beta.$

Note that $\xi\vert_{\mathbb{P}(Q_1)} = \mathcal{O}_{\mathbb{P}(Q_1)}$.
Hence we have
\begin{align*}
 \tilde{E}\cdot \gamma = \tilde{E}\cdot \beta^*\xi = 0.
\end{align*}
Now we have $N^1(Y)_{\mathbb{R}} = \Bigl\{ a\xi_1 + b\rho^*L_X \mid a,b\in \mathbb{R}\Bigr\}$ and $N^1\bigl(\mathbb{P}(E)\bigr)_{\mathbb{R}} = \Bigl\{ a\xi + b\pi^*L_X  \mid a,b\in \mathbb{R}\Bigr\}.$

We define $$\phi_1 : N^1\bigl(\mathbb{P}(E)\bigr)_{\mathbb{R}} \longrightarrow N^1(Y)_{\mathbb{R}}$$ by 
\begin{center}
 $ \phi_1( a\xi + b\pi^*L_X ) = a\xi_1 + b\rho^*L_X$.
\end{center}
This gives an isomorphism between real vector spaces $N^1\bigl(\mathbb{P}(E)\bigr)_{\mathbb{R}}$ and  $N^1(Y)_{\mathbb{R}}$.

Also we define \hspace{27mm} $U_1 : N^1(Y)_{\mathbb{R}} \longrightarrow N^1\bigl(\mathbb{P}(E)\bigr)_{\mathbb{R}}$
\begin{center}
 $U_1(\alpha) = \beta_*\eta^*(\alpha).$
\end{center}
In particular, $U_1(a\xi_1 + b\rho^*L_X) = a\xi + b\pi^*L_X.$ We construct an inverse for $U_1$.

Define $$D_1 : N^1\bigl(\mathbb{P}(E)\bigr)_{\mathbb{R}} \longrightarrow N^1(Y)_{\mathbb{R}}$$
\begin{center}
\hspace{8mm} $ D_1(\alpha) = \eta_*(\gamma\cdot\beta^*\alpha)$.
\end{center}
Note that 
$D_1(\xi) = \eta_*(\gamma\cdot\beta^*\xi) = \eta_*\bigl(\gamma\cdot(\eta^*\xi_1 + E)\bigr) = \eta_*\bigl(\gamma\cdot\eta^*\xi_1 + \gamma\cdot E\bigr) = \eta_*\gamma\cdot \xi_1 = \bigl[Y\bigr]\cdot \xi_1$ (see Proposition 3.1, \cite{F98}).
Similarly, $D_1(\pi^*L_X) = \rho^*L_X\cdot \bigl[Y\bigr]$. This shows that the maps $D_1$ and $\phi_1$ are the same maps. As the map $\gamma$ is nef and $\beta$ is dominant,   the map $D_1$ sends effective divisor in $\mathbb{P}(E)$ to $\overline{\Eff}^1(Y)$. Thus $\overline{\Eff}^1\bigl(\mathbb{P}(E)\bigr) \cong \overline{\Eff}^1\bigl(\mathbb{P}(E_1)\bigr)$. Inductively, we have $\overline{\Eff}^1\bigl(\mathbb{P}(E)\bigr) \cong \overline{\Eff}^1\bigl(\mathbb{P}(E_1)\bigr) \cong \cdots \cong \overline{\Eff}^1\bigl(\mathbb{P}(E_d)\bigr)$. Now the pseudo-effective cone $\overline{\Eff}^1\bigl(\mathbb{P}(E_d)\bigr)$ can be computed using Case (i). This completes the proof.
\end{proof}
\end{thm}
\begin{exm}
\rm
Note that by \cite{M21} we have $\Nef^1(\mathbb{P}_X(E))=\overline{\Eff}^1(\mathbb{P}_X(E))$ if and only if $\Nef^1(X)=\overline{\Eff}^1(X)$ under the assumption that $E$ is slope semistable with $c_2(\End(E))=0$.

Let $T_{\mathbb{P}^2}$ be the tangent bundle on $\mathbb{P}^2$. Let $\pi : \mathbb{P}(T_{\mathbb{P}^2})\longrightarrow \mathbb{P}^2$ be the associated projective bundle having Picard number 2.
Note that $T_{\mathbb{P}^2}$ is a stable rank 2 bundle and $c_2(\End(T_{\mathbb{P}^2}))\neq 0$. Hence $\overline{\Eff}^1(\mathbb{P}_{\mathbb{P}^2}(T_{\mathbb{P}^2})) \neq \Nef^1(\mathbb{P}_{\mathbb{P}^2}(T_{\mathbb{P}^2}))$ and $\overline{\Eff}^1(\mathbb{P}(T_{\mathbb{P}^2}))$
can not be computed using Theorem \ref{thm4.1} in this case. We next compute the pseudo-effective cone $\overline{\Eff}^1(\mathbb{P}_{\mathbb{P}^2}(T_{\mathbb{P}^2}))$ and show that $$\overline{\Eff}^1\bigl(\mathbb{P}(T_{\mathbb{P}^2}(-1))\bigr) = \Nef^1\bigl(\mathbb{P}(T_{\mathbb{P}^2}(-1))\bigr) $$
although $T_{\mathbb{P}^2}$ is stable with $c_2(\End(T_{\mathbb{P}^2})) \neq 0$.

Note that $\mathbb{P}_{\mathbb{P}^2}(T_{\mathbb{P}^2})$ is the universal hypersurface in $\mathbb{P}^2\times \mathbb{P}^{2\vee}$. Consider the two dominant projections to $\mathbb{P}^2$ and $\pi$ be one of them. Let $\pi^*H$ be the pullback of any hyperplane divisor class in $\mathbb{P}^2$ and $\xi-\pi^*H$ be the pullback of any line classes in $\mathbb{P}^2$. Our claim is that $\pi^*H$ and $\xi-\pi^*H$ generate the nef cone of $\mathbb{P}_{\mathbb{P}^2}(T_{\mathbb{P}^2})$.

We  also have the following Euler exact sequence
$$0\longrightarrow \mathcal{O}_{\mathbb{P}^2}(-1)\longrightarrow \mathcal{O}_{\mathbb{P}^2}^{\oplus 3}\longrightarrow T_{\mathbb{P}^2}(-1)\longrightarrow 0.$$

Thus $T_{\mathbb{P}^2}(-1)$ is generated by global sections fewer than $4$ global sections. Hence by [Remark 4.2 (ii), \cite{FL}] we have
$$\overline{\Eff}^1\bigl(\mathbb{P}(T_{\mathbb{P}^2}(-1))\bigr) = \Nef^1\bigl(\mathbb{P}(T_{\mathbb{P}^2}(-1))\bigr) = \Bigl\{y_0(\xi-\pi^*H)+y_1\pi^*H\mid y_0\geq 0, y_1\geq 0\Bigr\}.$$
\end{exm}

\section{Volume function of fiber product of projective bundles}
Recall that if $E$ is a vector bundle on a smooth complex curve $C$ having Harder-Narasimhan filtration
$$0=E_d\subsetneq E_{d-1}\subsetneq \cdots \subsetneq E_1\subsetneq E_0 =E$$
 with sucessive slopes $\mu_i$, then the  probability measure $\nu_E$ on $\mathbb{R}$ defined as follows:
$$\nu_E := \frac{1}{\rk(E)}\sum\limits_{i=0}^{d-1}\{\rk(E_{i})-\rk(E_{i+1})\}\delta_{\mu_i},$$
where $\delta_x$ is the Dirac measure concentrated on $\{x\}$. Also, for any real number $\varepsilon > 0$, let $T_{\varepsilon}$ be the operator on the set of all Borel probability measure on $\mathbb{R}$ which is defined by the following property.

If $\nu$ is a Borel probability measure and if $f$ is a continuous function of compact support on $\mathbb{R}$, then
$$\int\limits_{\mathbb{R}} f(x)T_{\varepsilon}\nu (dx) = \int\limits_{\mathbb{R}} f(\varepsilon x)\nu(dx).$$
For  a line bundle $M$ of degree $a$ on $C$, we then have
$$0=E_{d} \otimes M \subsetneq E_{d-1} \otimes M \subsetneq \cdots \subsetneq E_0 \otimes M = E\otimes M$$ is the Harder-Narasimhan filtration of $E\otimes M$ having succesive slopes $\mu_i+a$. This shows that \large$$\nu_{E\otimes M} = \tau_a\nu_E.$$\normalsize Here $\tau_a$ is defined as
$$\int\limits_{\mathbb{R}} f(x) \tau_a\nu (dx) = \int\limits_{\mathbb{R}} f(x+a) \nu (dx).$$

\baselineskip18pt

\begin{defi}\label{defi5.1}
 A sequence of Borel probability measure $\{\nu_n\}_{n\geq 1}$ is said to converge vaguely to a Borel measure $\nu$ if for any continuous function $f$ on $\mathbb{R}$ of compact support, one has $$\int\limits_{\mathbb{R}} f(x)\nu(dx) =\lim\limits_{n\to\infty}\int\limits_{\mathbb{R}} f(x)\nu_n(dx).$$
\end{defi}
We recall the following theorem from \cite{C}. In what follows, we denote the function field of a smooth irreducible complex projective variety $C$ by $K := k(C)$.
\begin{thm}\label{thm2.1}
 Let $\pi:X\longrightarrow C$ be a projective and flat morphism, and $L$ be an arbitrary line bundle on $X$ such that $L_{K}$ is big. Then 
 
 (1) The sequence of Borel probability measures $\bigl(T_{\frac{1}{n}}\nu_{\pi_*(L^{\otimes n})}\bigr)_{n\geq 1}$ converges vaguely to a Borel measure $\nu^{\pi}_L$.
 
 (2) the following equality holds 
 $$\vol_X(L) = \dim(X) \vol_{X_K}(L_K) \int\limits_{\mathbb{R}} x_+\nu^{\pi}_L(dx),$$
where $x_+ = \max\{x,0\}$.
\end{thm}
We observe that if $E$ is a semistable bundle on $C$, and  $\pi : \mathbb{P}(E)\longrightarrow C$ is the projection map, then $T_{\frac{1}{n}}\nu_{\pi_*(\mathcal{O}_{\mathbb{P}(E)}(n))}$ converges vaguely to $\nu_E$ as $n \to\infty$. Hence $\nu^{\pi}_{\mathcal{O}_{\mathbb{P}(E)}(1)} = \nu_{E}$ in this case.

Now recall the following commutative fibre product diagram:
\begin{center}
 \begin{tikzcd} 
X =\mathbb{P}(F)\times_C \mathbb{P}(E) = \mathbb{P}(\psi^*E) \arrow[r, "\pi_2"] \arrow[d, "\pi_1"]
& \mathbb{P}(E)  \arrow[d,""]\\
 \mathbb{P}(F) \arrow[r, "\psi" ]
& C
\end{tikzcd}
\end{center}
Recall that $\pi = \psi\circ \pi_1$ and $\mathcal{E} = \psi^*E$.
Hence $\Pic\bigl(\mathbb{P}(\mathcal{E})\bigr) = \mathbb{Z}\cdot\mathcal{O}_{\mathbb{P}(\mathcal{E})}(1) \oplus \pi_1^*\Pic(\mathbb{P}(F))$ and $\Pic\bigl(\mathbb{P}(F)\bigr) = \mathbb{Z}\cdot\mathcal{O}_{\mathbb{P}(F)}(1)\oplus \psi^*\Pic(C).$
Let $L = \mathcal{O}_{\mathbb{P}(\mathcal{E})}(m) \otimes \pi_1^*\mathcal{O}_{\mathbb{P}(F)}(l)\otimes \pi_1^*\psi^*\mathcal{M}$ be a line bundle on $\mathbb{P}(\mathcal{E})$ for some line bundle $\mathcal{M}$ of degree $a$ on $C$ such that $L_{K}$ is big on $X_K$.
Note that $X_K \cong \mathbb{P}_K^{e-1}\times\mathbb{P}_K^{f-1}$ and hence $\Pic(X_K)  \cong \mathbb{Z}\times \mathbb{Z}$. Then $L_K$ corresponds to the line bundle of type $(m,l)\in \mathbb{Z}\times \mathbb{Z}$. As $L_K$ is big, we conclude $m>0,l>0$.
In particular, if $L$ is big, then $L_{K}$ is also big. Let us consider the cone
$S = \Bigl\{[L]\in N^1\bigl(\mathbb{P}(\mathcal{E})\bigr)_{\mathbb{R}} \mid  L\in \Pic\bigl(\mathbb{P}(\mathcal{E})\bigr), L_K$ is big on $X_K \Bigr\}$
$$= \Bigl\{[L]\in N^1\bigl(\mathbb{P}(\mathcal{E})\bigr)_{\mathbb{R}} \mid L = \mathcal{O}_{\mathbb{P}(\mathcal{E})}(m) \otimes \pi_1^*\mathcal{O}_{\mathbb{P}(F)}(l)\otimes \pi^*\mathcal{M}\in \Pic\bigl(\mathbb{P}(\mathcal{E})\bigr), m>0,l>0\Bigr\}.$$
Hence,  we have
\begin{center}
Big$\bigl(\mathbb{P}(\mathcal{E})\bigr) \subsetneq S \subsetneq N^1\bigl(\mathbb{P}(\mathcal{E})\bigr)_{\mathbb{R}}.$
\end{center}
Note that $\nu^{\pi}_{L\otimes \pi^*M} = \tau_a\nu^{\pi}_L$ for any line bundle $L$ and $M$ with $M$ having degree $a$ on $C$.

Thus by Theorem \ref{thm2.1} we have
\begin{align*}
& \vol_X(L)
\\
&= \dim(X) \vol_{X_K}(L_K) \int x_+\nu^{\pi}_L(dx)
\\
& = \dim(X) \vol_{X_K}(L_K) \int x_+ \tau_a \nu^{\pi}_{\mathcal{O}_{\mathbb{P}(\mathcal{E})}(m) \otimes \pi_1^*\mathcal{O}_{\mathbb{P}(F)}(l)} (dx)
\\
&=\dim(X) \vol_{X_K}(L_K) \int (x+a)_+\nu^{\pi}_{\mathcal{O}_{\mathbb{P}(\mathcal{E})}(m) \otimes \pi_1^*\mathcal{O}_{\mathbb{P}(F)}(l)} (dx)
\\
&= \dim(X) \vol_{X_K}(L_K) \lim_{n\to \infty} \int (x+a)_+T_{\frac{1}{n}}\nu_{\pi_*(\mathcal{O}_{\mathbb{P}(\mathcal{E})}(m) \otimes \pi_1^*\mathcal{O}_{\mathbb{P}(F)}(l))^{\otimes n}} (dx)
\\
&=\dim(X) \vol_{X_K}(L_K) \lim_{n\to \infty} \int (x+a)_+T_{\frac{1}{n}}
\nu_{\Sym^{mn}(E)\otimes \Sym^{ln}(F)} (dx)
\end{align*}
\baselineskip=18pt
Next we compute the Borel measure $ \nu_{\Sym^{mn}(E)\otimes \Sym^{ln}(F)}$. This will involve the computation of the Harder-Narasimhan filtration of $\Sym^{mn}(E)\otimes \Sym^{ln}(F)$.
\section{Computation of the volume function}\label{Sec5}
Firstly we determine the Harder-Narasimhan filtration of $E\otimes F$ in terms of the Harder-Narsimhan filtrations of $E$ and $F$. The results in this section are inspired by \cite{C}.
\subsection{Harder-Narasimhan filtration of tensor product}\label{subsec6.1}  Recall that
\begin{align}\label{seq3}
 0=E_d\subsetneq E_{d-1}\subsetneq \cdots \subsetneq E_1\subsetneq E_0 =E
\end{align}
 and 
\begin{align}\label{seq4}
 0=F_h\subsetneq F_{h-1}\subsetneq F_{h-2}\subsetneq\cdots\subsetneq F_{1} \subsetneq F_0 = F
\end{align}
 are the Harder-Narsimhan filtrations for $E$ and $F$ respectively. We fix the notations $Q_i = E_i/E_{i+1}$ for each $i \in \{0,1,2,\cdots,d-1\}$ and $Q'_j = F_j/F_{j+1}$ for each $j\in \{0,1,2,\cdots,h-1\}$.

Let $J_1 = \{0,1,2,\cdots,d-1\}$ and $J_2 = \{ 0,1,2,\cdots,h-1\}$. Consider the set $\Theta = J_1\times J_2$. 
\linebreak We define a partial order on $\Theta$ as follows:
$$ (a,b)\leq (c,d) \iff a\leq c, b\leq d.$$
We say a subset $A\subseteq \Theta$ is saturated if $\alpha \in A, \beta \in \Theta$ and $\beta \geq \alpha \implies \beta \in A$.
For a subset $A\subseteq \Theta$ we define the saturation of $A$,  denoted by $\overline{A}$ as follows:
\begin{center}
$\overline{A} = \bigl\{ \beta \in \Theta \mid \exists \alpha \in A$ such that $\beta \geq \alpha\bigr\}$.
\end{center}
For a nonempty subset $A\subseteq \Theta$ and $\alpha =(a,b) \in A$, we define 
\begin{align}\label{seq5}
\Im_{\alpha} : = E_a\otimes F_b \hspace{2mm} \rm and \hspace{2mm}\Im_{A} := \sum\limits_{\alpha\in A} \Im_{\alpha}.
\end{align}
i.e. $\Im_A$ is the sum of subbundles generated by $\Bigl\{\Im_{\alpha} : \alpha \in A\Bigr\}.$
Write by convention $E_{\emptyset} = 0$.
Then $\Im_{A} = \Im_{\overline{A}}$. Also for any two subsets $A_1 \subseteq A_2\subseteq \Theta$, we have $$\Im_{A_1} \subseteq \Im_{A_2}.$$

For any $\alpha =(a,b)\in \Theta$, we define 
\begin{align}\label{seq6}
\mathcal{Q}_\alpha = Q_a\otimes Q'_b.
\end{align}
Since the ground field $\mathbb{C}$ has charactersitic 0, we have  $\mathcal{Q}_{\alpha}$ is a semistable bundle for each $\alpha \in \Theta$.
\begin{prop}\label{prop3.1}
 Assume that $A\subseteq \Theta$ is non-empty and saturated, and that $A''$ is a subset of $A$ consisting of minimal elements. Then $A':=A\setminus A''$ is also saturated. Furthermore, one has an isomorphism
 \begin{align}\label{isom}
 \Im_A/\Im_{A'} \cong \bigoplus\limits_{\alpha \in A''}\mathcal{Q}_{\alpha}
 \end{align}
 \begin{proof}
 Let $\alpha\in A'$ , $\beta\in \Theta$, and $\beta\geq \alpha$. As $A$ is saturated, we have $\beta\in A$. But $\beta$ is not a minimal element of $A$. Thus $\beta \in A'$. This shows that $A'$ is saturated.

 Let $A=\bigl\{\alpha_1,\alpha_2,\cdots,\alpha_n\bigr\}$ such that
 $A''=\bigl\{\alpha_1,\alpha_2,\cdots,\alpha_m\bigr\}$ and $A'=\bigl\{\alpha_{m+1},\alpha_{m+2},\cdots,\alpha_{n}\bigr\}$.

  Recall that for any $\alpha = (a,b)\in A$, we define $\Im_{\alpha} = E_a\otimes F_b$ and $\mathcal{Q}_{\alpha} = Q_a\otimes Q'_b$. We note that  there is a surjection $$\psi_{\alpha} : \Im_{\alpha} = E_a\otimes F_b  \longrightarrow Q_a\otimes F_b\longrightarrow \mathcal{Q}_{\alpha} = Q_a\otimes Q'_b\longrightarrow 0.$$
 Define $$\Psi : \Im_{A} =\sum\limits_{\alpha \in A} \Im_{\alpha} \longrightarrow \bigoplus\limits_{\alpha\in A''}\mathcal{Q}_{\alpha}$$
 as follows :
 $$\Psi(\sum\limits_{j=1}^nx_j) = \bigl(\psi_{\alpha_1}(x_1),\psi_{\alpha_2}(x_2),\cdots,\psi_{\alpha_m}(x_m)\bigr),$$
 where $x_j\in \Im_{\alpha_j}$ for each $j\in \{1,2,\cdots,n\}$. We fix the notation $\alpha_j = (a_j,b_j)$ for $j$ satisfying $1\leq j \leq n$. Note that if $\alpha \leq \beta$ for some $\alpha,\beta\in A$, then we have $\Im_{\beta}\hookrightarrow \Im_{\alpha}$. Moreover, if $\alpha \leq \beta$ for some $\alpha,\beta\in A$ and $\alpha\neq \beta$, then $\Im_{\beta} \subseteq \ker(\psi_{\alpha})$.

 First we claim the following : if $\sum\limits_{j=1}^nx_j = 0\in \Im_A$, then $$\Psi(\sum\limits_{j=1}^nx_j) = \bigl(\psi_{\alpha_1}(x_1),\psi_{\alpha_2}(x_2),\cdots,\psi_{\alpha_m}(x_m)\bigr) = (\underbrace{0,0,0,\cdots,0}_{m-times}) \in \bigoplus\limits_{\alpha\in A''}\mathcal{Q}_{\alpha}.$$

  Now suppose $\psi_{\alpha_i}(x_i) \neq 0$ for some $\alpha_i$ with $1\leq i \leq m$. Consider the set
 \begin{center}
  $T:=\Bigl\{ l \mid 1\leq l \leq n, (a_l,b_l) \geq (a_i,b_i)$ and $(a_l,b_l) \neq (a_i,b_i)\Bigr\}.$
  \end{center}

 We then have $$0=\sum\limits_{j=1}^n x_j = x_i + \sum\limits_{l\in T}x_l+\sum\limits_{l\notin T, l\neq i}x_l$$
Note that for any $l\in T$, we have $\Im_{(a_l,b_l)}\subseteq \ker(\psi_{\alpha_i})$. Thus $\sum\limits_{l\in T}x_l \in \sum\limits_{l\in T}\Im_{(a_l,b_l)} \subseteq \ker(\psi_{\alpha_i}).$

Therefore $$x_i+\sum\limits_{l\notin T, l\neq i}x_l \in \ker(\psi_{\alpha_i}) \subseteq \Im_{\alpha_i}.$$
We will now show that $\sum\limits_{l\notin T, l\neq i}x_l \in \ker(\psi_{\alpha_i})$ and this will contradict the fact that $x_i \notin \ker(\psi_{\alpha_i}).$

Note that $$\sum\limits_{l\notin T, l\neq i}x_l \in \Im_{(a_i,b_i)}.$$
Also $$\sum\limits_{l\notin T, l\neq i}x_l  \in \sum\limits_{l\in T,l\neq i}\Im_{(a_l,b_l)} \subseteq \Im_{(s,t)},$$
where $s=\min\{a_l \mid 1\leq l \leq n, l\notin T, l\neq i\}$ and $t= \min\{b_l \mid 1\leq l \leq n, l\notin T, l \neq i\}.$ Let $s=a_p$ for some $p\notin T$ and $p\neq i$, and $t=b_q$ for some $q\notin T$ and $q\neq i$.

 Then we observe that
 $$\sum\limits_{l\notin T, l\neq i}x_l \in \Im_{(a_i,b_i)} \cap \Im_{(a_p,b_q)} = \Im_{(\max\{a_p,a_i\},\max\{b_q,b_i\})}.$$

 Next we show that
 $(\max\{a_p,a_i\},\max\{b_q,b_i\}) \neq (a_i,b_i)$. If not, then $\max\{a_p,a_i\} = a_i$ and $\max\{b_q,b_i\}= b_i$. This implies $a_i\geq a_p$ and $b_i \geq b_q$. Since $p\notin T$ and $p\neq i$, we have either $a_p< a_i$ or $b_p<b_i$. Consider the following two cases :
 \begin{enumerate}
  \item Case 1 : If $a_p < a_i$, then we arrive at a contradiction as $a_i\geq a_p.$
  \item Case 2 : If $b_p < b_i$,  then  by definition of $b_q$ we have
  $b_q \leq b_p < b_i$, and this also gives a contradiction.
 \end{enumerate}


Therefore $(\max\{a_p,a_i\},\max\{b_q,b_i\}) \geq (a_i,b_i)$ and $(\max\{a_p,a_i\},\max\{b_q,b_i\}) \neq (a_i,b_i).$

Thus we conclude $$\sum\limits_{l\notin T,l\neq i}x_l \in \Im_{(\max\{a_p,a_i\},\max\{b_q,b_i\})} \subseteq \ker(\psi_{\alpha_i}).$$
This completes the proof of our claim.

 Now we show that $\Psi$ is well-defined. Let $\sum\limits_{j=1}^nx_j$ and $ \sum\limits_{j=1}^nx'_j \in \Im_A$ such that $\sum\limits_{j=1}^nx_j= \sum\limits_{j=1}^nx'_j$. Then $$\sum\limits_{j=1}^n\bigl(x_j-x'_j) = \underbrace{0+0+0+\cdots+0}_{n-times}   \in \Im_A.$$ Hence by the previous observation, we get $\psi_{\alpha_i}(x_j-x'_j) = 0$ for each $i$ and $j$ with $1\leq i,j \leq m$. As each $\psi_{\alpha}$ is a homomorphism, we have $\psi_{\alpha_i}(x_j) = \psi_{\alpha_i}(x'_j)$ for each $i$ and $j$ satisfying $1\leq i,j \leq m$. Therefore we conclude $\Psi\bigl(\sum\limits_{j=1}^nx_j\bigr) = \Psi\bigl(\sum\limits_{j=1}^nx'_j \bigr).$ This proves the well-definedness of the map $\Psi$.

 Clearly, we have $\Im_{A'} \subseteq $ ker$(\Psi)$. Also, the map $\Psi$ is surjective as each $\psi_{\alpha}$ is surjective for any $\alpha \in A^{''}$.
Next we will show that ker$(\Psi) \subseteq \Im_{A'}$.




  Let $\alpha = (a,b) \in A''$ be a minimal element and $x\in$ Ker$(\psi_{\alpha})$. For any minimal element $\alpha = (a,b) \in A^{''}$, we have the following commutative diagram of short exact sequences:

\[
\begin{tikzcd}
  & 0 \arrow[d] & 0 \arrow[d] & 0 \arrow[d]& \\
  0 \arrow[r] & E_{a+1}\otimes F_{b+1} \arrow[d, "\eta_1"] \arrow[r, "\theta_1"] & E_{a}\otimes F_{b+1} \arrow[d, "\beta_1"] \arrow[r, "\theta_2"] & Q_{a}\otimes F_{b+1}\arrow[d, "\gamma_1"] \arrow[r] & 0 \\
  0 \arrow[r] & E_{a+1}\otimes F_{b} \arrow[r, "\theta_3"]\arrow[d, "\eta_2"] & E_{a}\otimes F_{b} \arrow[r, "\theta_4"]\arrow[d,"\beta_2"] & Q_{a}\otimes F_{b} \arrow[d, "\gamma_2"] \ar[r] & 0\\
  0 \arrow[r] & E_{a+1}\otimes Q'_{b} \arrow[r, "\theta_5"] \arrow[d] & E_{a}\otimes Q'_{b} \arrow[r, "\theta_6"] \arrow[d] & Q_{a}\otimes Q_{b'} \ar[r] \arrow[d] & 0\\
  & 0 & 0 & 0 &
\end{tikzcd}
\]

Note that the map $\psi_{\alpha} = \gamma_2\circ \theta_4$. Let $x\in $ Ker$(\psi_{\alpha})$. Then $(\gamma_2\circ \theta_4)(x) = \gamma_2(\theta_4(x))=0$.
So there exists $y\in Q_a\otimes F_{b+1}$ such that $\gamma_1(y) = \theta_4(x)$. Now there exists  $z \in E_a\otimes F_{b+1}$ such that $\theta_2(z) = y$. This implies $\gamma_1(\theta_2(z))= \theta_4(x)$. By the commutativity of the diagram, we have $\gamma_1(\theta_2(z)) = \theta_4(\beta_1(z)) = \theta_4(x)$. This implies $x- \beta_1(z)  \in $ Ker$(\theta_4)$. Thus there exists $t\in E_{a+1}\otimes F_b$ such that $x- \beta_1(z) = \theta_3(t)$. In other words, $x=\beta_1(z)+\theta_3(t)$. This shows that Ker$(\psi_{\alpha}) \subseteq E_{a+1}\otimes F_b + E_{a}\otimes F_{b+1} \subseteq \Im_{A'}$.

Let $\sum\limits_{i=1}^n x_i \in $ Ker$(\Psi)$. Thus $\psi_{\alpha_i}(x_i) = 0$ for every $i\in\{1,2,\cdots,m\}$, i.e. $x_i\in $ Ker$(\psi_{\alpha_i}).$ Hence by the previous observation  we have $\sum\limits_{i=1}^nx_i\in \Im_{A'}$.
 Therefore we conclude that Ker$(\Psi) \subseteq \Im_{A'}$.

We then have the required isomorphism :
 \begin{align*}
  \Psi: \Im_A/\Im_{A'} \cong \bigoplus\limits_{\alpha\in A''}\mathcal{Q}_{\alpha}
 \end{align*}

 \end{proof}
\end{prop}
For any $\alpha =(a,b) \in \Theta$, let $\mu_{\alpha} =\mu(Q_a)+\mu(Q'_b)$. With this notation, $\mu_{\alpha}$ is the slope of $\mathcal{Q}_{\alpha}$, where $\mathcal{Q}_{\alpha}$ is defined as in (\ref{seq6}). Denote by $\Sigma$  the set of all real numbers of the form $\mu_{\alpha}$, where $\alpha$ ranges through vectors in $\Theta$. Note that, from the two Harder-Narsimhan filtrations (\ref{seq3}) and (\ref{seq4}), we have $\alpha \leq \beta $ implies $\mu_{\alpha} \leq \mu_{\beta}$, and $\alpha \lneqq \beta $ implies $\mu_{\alpha} \lneqq \mu_{\beta}$.
\begin{thm}\label{thm3.2}
 The set $\Sigma$ identifies with that of succesive slopes of $E\otimes F$. Furthermore, suppose that the elements in $\Sigma$ are ordered as 
 $$w_0<w_1<w_2<\cdots<w_{t-1},$$
then 
\begin{align}\label{seq7}
0=\Im_{B_t} \subsetneq \Im_{B_{t-1}} \subsetneq \cdots \subsetneq\Im_{B_1}
\subsetneq \Im_{B_0}
\end{align}
is the Harder-Narasimhan filtration of $E\otimes F$, where $$B_j =\Bigl\{\alpha \in \Theta \mid \mu_{\alpha}\geq w_j\Bigr\} \subseteq \Theta$$ for each $j \in \{ 0,1,\cdots,t-1\}$ and $B_t =\emptyset$ by convention.
\begin{proof}
 We have $B_j\setminus B_{j+1} = \{\alpha \in \Theta \mid \mu_{\alpha} = w_j\}$. Moreover if $\mu_{\alpha}=w_j$, then $\alpha$ is a minimal element of $B_j$. Therefore by the previous Proposition \ref{prop3.1} we get $$\Im_{B_j}/ \Im_{B_{j+1}} \cong \bigoplus\limits_{\alpha\in B_j\setminus B_{j+1}} \mathcal{Q}_{\alpha}$$
is semistable of slope $w_j$.
\end{proof}
\end{thm}
 \begin{xrem}
  \rm In the above Theorem \ref{thm3.2}, the rank of the succesive semistable quotients are
  $$\mathcal{R}_j := \rk\bigl(\Im_{B_j}/ \Im_{B_{j+1}}\bigr) = \sum\limits_{\alpha \in B_j\setminus B_{j+1}} \rk\bigl(\mathcal{Q}_{\alpha}\bigr) = \sum\limits_{\alpha = (a,b) \in B_j\setminus B_{j+1}} \rk(Q_a)\cdot \rk(Q'_b).$$
 \end{xrem}

\subsection{Harder Narasimhan Filtration for $\Sym^{mn}(E)$ :}
Recall that  
\begin{align}\label{seq8}
 0=E_d\subsetneq E_{d-1}\subsetneq E_{d-2}\subsetneq\cdots\subsetneq E_{1} \subsetneq E_0 = E
\end{align}
is the Harder-Narasimhan filtration of $E$, and $Q_i=E_i/E_{i+1}$ for each $i$. Let $\mu_i = \mu(Q_i)$ so that we have $$\mu_{d-1} > \mu_{d-1} > \cdots >\mu_1 > \mu_0.$$
Let $\Phi = \{0,1,\cdots,d-1\}^{mn}$. We define a partial order $\leq$ on $\Phi$ as follows:

For $\alpha = (a_1,a_2,\cdots,a_{mn})$ and $\beta = (b_1,b_2,\cdots,b_{mn}) \in \Phi$, we define $$\alpha \leq \beta  \iff a_i\leq b_i \hspace{2mm} \text{for}\hspace{1mm} \text{every} \hspace{1mm} i =1,2,\cdots,mn.$$
For $\alpha = (a_i)_{i=1}^{mn}\in \Phi$, we define $E_{\alpha} = \bigotimes\limits_{i=1}^{mn}E_{a_i}$. For a  subset $A\subseteq \Phi$, we define $E_A : = \sum\limits_{\alpha\in A}E_{\alpha}.$

Partitions of an integer are unordered, but the $a_0,\cdots,a_{d-1}$ are
ordered (meaning their order matters). The ordered ones are called
\it compositions. \rm

The symmetric group $S^{mn}$ acts on $\Phi$. Let $\tilde{\Phi} := \Phi/S^{mn}$.  Note that the elements in $\tilde{\Phi}$ are in bijective correspondence with the set of compositions of $mn$ into sum of $d$ non-negative integers $(a_0,a_1,\cdots,a_{d-1})$, i.e. $\sum\limits_{i=0}^{d-1}a_i = mn$. Any element $[\alpha]$ in $\tilde{\Phi}$ will be denoted by $\underline{a} = (a_0,a_1,\cdots,a_{d-1})$ under this correspondence. For such  a class $[\alpha]$, we use the symbol $\tilde{Q}_{[\alpha]}(E)$ to denote the tensor product $\bigotimes\limits_{i=0}^{d-1}\Sym^{a_i}Q_i$.

Let $A\subseteq \Phi$ be such that $A$ is invariant under the action of $S^{mn}$. Then we have a map $$\theta : E_A \longrightarrow \Sym^{mn}E.$$ 
Define $S_AE:=$ Im$(\theta)$. This gives $E_A \rightarrow \hspace{-8pt} \rightarrow S_AE$ and hence $E_A/S^{mn} \cong S_AE$.

For $\alpha = (a_i)_{i=1}^{mn}\in \Phi$, we define $\mu_{\alpha} : = \sum\limits_{i=1}^{mn} \mu_{a_i}$. Let $\Sigma_{mn}$ be the set of all $\mu_{\alpha}$ for $\alpha \in \Phi$. We order the elements in $\Sigma_{mn}$ as follows:
$$v_0<v_1<\cdots<v_{k-1}.$$

Let $A_j = \{ \alpha \in \Phi \mid \mu_\alpha \geq v_j\}$ for $j\in\{0,1,\cdots,k-1\}$. We further have by assumption that $A_{k} =\emptyset$.

Then by [Proposition 3.2, \cite{C}] the following filtration
\begin{align}\label{seq9}
 0 = S_{A_k}E \subsetneq S_{A_{k-1}}E \subsetneq \cdots \subsetneq S_{A_1}E \subsetneq S_{A_0}E=S_AE
\end{align}
 is the Harder-Narasimhan filtration of $\Sym^{mn}(E)$ such that $$S_{A_j}E/S_{A_{j+1}}E \cong \bigoplus\limits_{[\alpha] \in (A_j\setminus A_{j+1})/S^{mn}} \tilde{Q}_{[\alpha]}(E)$$ are semistable bundles of slope $v_j$ for $j\in\{0,1,\cdots,k-1\}$.
 
 We also have the expression for the rank $R_j$ of each of the succesive quotients $S_{A_j}E/S_{A_{j+1}}E$ in (\ref{seq9}).
Take an arbitrary class $[\alpha]$ in $\tilde{\Phi}$ which correspondence to the composition $\underline{a} = (a_i)_{i=0}^{d-1}$ of $mn$. Then we have value $\mu_{\alpha} = \sum\limits_{i=0}^{d-1}a_i\mu_i$ where $\mu_{i} = \mu(Q_i)$ as in (\ref{seq8}). For each $i \in \{0,1,\cdots,d-1\}$, let $r_i$ be the rank of $Q_i$. Then the rank of $\tilde{Q}_{[\alpha]}(E)$ is given as follows :

$$r_{\underline{a}}:= \rk\bigl(\tilde{Q}_{[\alpha]}(E)\bigr) = \rk\Bigl( \bigotimes\limits_{i=0}^{d-1}\Sym^{a_i}Q_i\Bigr) = \prod\limits_{i=0}^{d-1}\binom{a_i+r_i-1}{{r_i}-1}.$$
Thus we obtain
\begin{align}\label{rk1}
R_j =\rk(S_{A_j}E/S_{A_{j+1}}E) \hspace{3mm} = \sum\limits_{\substack{\underline{a} =(a_i)_{i=0}^{d-1}\in \mathbb{N}^d,\\ a_0+a_1+\cdots+a_{d-1}=mn, \\ a_0\mu_0+a_1\mu_1+\cdots+a_{d-1}\mu_{d-1} = v_a}} r_{\underline{a}} \hspace{3mm} = \sum\limits_{\substack{\underline{a} =(a_i)_{i=0}^{d-1}\in \mathbb{N}^d,\\ a_0+a_1+\cdots+a_{d-1}=mn, \\ a_0\mu_0+a_1\mu_1+\cdots+a_{d-1}\mu_{d-1} = v_a}} \prod\limits_{i=0}^{d-1}\binom{a_i+r_i-1}{{r_i}-1}
\end{align}
\subsection{Harder Narasimhan Filtration for $\Sym^{ln}(F)$ :}
Recall that 
\begin{align}\label{seq10}
 0=F_h\subsetneq F_{h-1}\subsetneq F_{h-2}\subsetneq\cdots\subsetneq F_{1} \subsetneq F_0 = F
\end{align}
is the Harder-Narasimhan filtration of $F$, and $Q'_j=F_j/F_{j+1}$ for each $j$ with rank $\rk(Q'_j) =r'_j$. Let $\mu'_j = \mu(Q_j)$ so that we have $$\mu'_{h-1} > \mu'_{h-2} > \cdots >\mu'_1 > \mu'_0.$$

Similar steps can be taken  to get the Harder-Narasimhan filtration for $\Sym^{ln}(F)$, and the rank of the succesive slopes of the Harder-Narasimhan filtration of  $\Sym^{ln}(F)$.

Let $\Phi' = \{0,1,2,\cdots,h-1\}^{ln}$. For $\beta = (b_i)_{i=1}^{ln}\in \Phi'$, we define $F_{\beta} := \bigotimes\limits_{i=1}^{ln}F_{b_i}$, and for a subset $B\subseteq \Phi'$, we define $F_{B} := \sum\limits_{\beta\in B}F_{\beta}$.

Let $\Sigma_{ln} := \bigl\{\mu_{\beta} \mid \beta\in \Phi'\bigr\}$, where $\mu_{\beta} = \sum\limits_{i=1}^{ln} \mu_{b_i}$ if $\beta = (b_i)_{i=1}^{ln}\in \Phi'$. We order the elements in $\Sigma_{ln}$ as follows :
$$v'_0<v'_1<\cdots<v'_{p-1}$$
For $j\in\{0,1,\cdots,p-1\}$, we define $A'_j := \bigl\{\beta \in \Phi'\mid \mu_{\beta} \geq v'_j\bigr\}$ and $A'_p = \emptyset$.
Let
\begin{align}\label{seq11}
0 = S_{A'_p}F \subsetneq S_{A'_{p-1}}F \subsetneq \cdots \subsetneq S_{A'_1}F \subsetneq S_{A'_0}F=S_A'F
\end{align}
be the Harder-Narasimhan filtration of $\Sym^{ln}(F)$ with sucessive semistable quotients $$S_{A'_j}F/S_{A'_{j+1}}F \cong \bigoplus\limits_{[\beta] \in (A'_j\setminus A'_{j+1})/S^{ln}} \tilde{Q}_{[\beta]}(F)$$
having slope $v'_j$ for each $j\in \{0,1,\cdots,p-1\}$.

Then
\begin{align}\label{rk2}
 R'_j =\rk(S_{A'_j}F/S_{A'_{j+1}}F)\hspace{3mm} = \sum\limits_{\substack{\underline{a}' =(a'_i)_{i=0}^{h-1}\in \mathbb{N}^h, \\ a'_0+a'_1+\cdots+a'_{h-1}=ln, \\ a'_0\mu'_0+a'_1\mu'_1+\cdots+a'_{h-1}\mu'_{h-1} = v'_b}} r_{\underline{a}'}  \hspace{3mm} = \sum\limits_{\substack{\underline{a}' =(a'_i)_{i=0}^{h-1}\in \mathbb{N}^h, \\ a'_0+a'_1+\cdots+a'_{h-1}=ln, \\ a'_0\mu'_0+a'_1\mu'_1+\cdots+a'_{h-1}\mu'_{h-1} = v'_b}} \prod\limits_{j=0}^{h-1}\binom{a'_j+r'_j-1}{{r'_j}-1}
 \end{align}
\subsection{Harder Narasimhan Filtration for $\Sym^{mn}(E) \otimes \Sym^{ln}(F)$ :}
We will use the notations and conventions used in  subsection \ref{subsec6.1} from now on to get the Harder-Narsimhan filtration of the tensor product $\Sym^{mn}(E) \otimes \Sym^{ln}(F)$.

Recall that 
\begin{align}\label{seq14}
 0 = S_{A_k}E \subsetneq S_{A_{k-1}}E \subsetneq \cdots \subsetneq S_{A_1}E \subsetneq S_{A_0}E=S_AE
\end{align}
and 
\begin{align}\label{seq13}
0 = S_{A'_p}F \subsetneq S_{A'_{p-1}}F \subsetneq \cdots \subsetneq S_{A'_1}F \subsetneq S_{A'_0}F=S_A'F
\end{align}
are the Harder Narasimhan filtrations of $\Sym^{mn}(E)$ and $\Sym^{ln}(F)$ respectively as in (\ref{seq9}) and (\ref{seq11}).

Recall that the succesive quotients $S_{A_j}E/S_{A_{j+1}}E$ in (\ref{seq14}) have slope $v_j$ for $j\in\{0,1,2,\cdots,k-1\}$ and similarly the succesive quotients $S_{A'_j}F/S_{A'_{j+1}}F$ in (\ref{seq13}) have slope $v'_j$ for $j\in \{0,1,2,\cdots,p-1\}$ respectively.

Let $J_1=\{0,1,2,\cdots,k-1\}$ and $J_2=\{0,1,2,\cdots,p-1\}$. Consider the product $\Theta = J_1 \times J_2$. 

Then for $\alpha = (a,b) \in \Theta$, we have $$\Im_{\alpha} = S_{A_a}E \otimes S_{A'_b}F.$$ We denote $\mathcal{Q}_{\alpha} = (S_{A_a}E/S_{A_{a+1}}E)\otimes (S_{A'_b}F/S_{A'_{b+1}}F).$ Then $\mathcal{Q}_{\alpha}$ is semistable. 
\vspace{2mm}

Let $$\Sigma = \Bigl\{ \mu_{\alpha}\mid  \alpha \in \Theta\Bigr\},$$ where $\mu_{\alpha} := \mu(\mathcal{Q}_{\alpha}) = \mu\bigl(S_{A_a}E/S_{A_{a+1}}E\bigr) + \mu\bigl(S_{A'_b}F/S_{A'_{b+1}}F\bigr) =  v_a+v'_b$ for $\alpha = (a,b) \in \Theta$. 
\vspace{2mm}

We order the elements of $\Sigma$ as follows:
$$w_0<w_1<\cdots<w_{q-1}.$$
Let $B_j =\{\alpha \in \Theta \mid \mu_{\alpha} \geq w_j\}$ for $j\in\{0,1,\cdots,q-1\}$, with $B_q=\emptyset$. Then by Theorem \ref{thm3.2} we have
\begin{align}\label{seq12}
\Im_{B_{q-1}} \subsetneq \cdots \subsetneq \Im_{B_{1}}\subsetneq \Im_{B_0}
\end{align}
is the Harder-Narasimhan filtration of $\Sym^{mn}(E)\otimes \Sym^{ln}(F)$ such that $\Im_{B_j}/\Im_{B_{j+1}} \cong \bigoplus\limits_{\alpha \in B_j\setminus B_{j+1}}\mathcal{Q}_{\alpha}$ is semistable having slope $w_j$ for each $j\in\{0,1,\cdots,q-1\}$.

Note that $\rk(\mathcal{Q}_{\alpha}) = R_a\cdot R'_b,$ where $\alpha =(a,b) \in \Theta$, and $R_a$ and $R'_b$ are as in (\ref{rk1}) and (\ref{rk2}). Therefore, for $j\in\{0,1,\cdots,q-1\}$, we have $$\mathcal{R}_j:= \rk(\Im_{B_j}/\Im_{B_{j+1}}) = \sum\limits_{\alpha =(a,b) \in B_j\setminus B_{j+1}}R_a\cdot R'_b = \sum\limits_{\substack{\alpha = (a,b) \in \Theta,\\ \mu_{\alpha}=w_j}}R_a\cdot R'_b = \sum\limits_{v_a+v'_b =w_j}R_a\cdot R'_b.$$

\subsection{Expression for $\nu_{\Sym^{mn}(E)\otimes \Sym^{ln}(F)}$}
Therefore by definition we have
\begin{align}\label{eq1}
\nu_{\Sym^{mn}(E)\otimes \Sym^{ln}(F)} = \frac{1}{\rk\bigl(\Sym^{mn}(E)\otimes \Sym^{ln}(F)\bigr)} \sum\limits_{j=0}^{q-1}\mathcal{R}_j\delta_{w_j}
\end{align}
where $\delta_{w_j}$ is the Dirac measure concentrated at $w_j$.
\vspace{2mm}

Next we compute the term $\sum\limits_{j=0}^{q-1}\mathcal{R}_j\delta_{w_j}.$ Recall that $\rk(E) =e$ and $\rk(F) =f$. Consider the vector $$\underline{s} := (s_1,s_2,\cdots,s_e,s'_1,s'_2,\cdots,s'_f) \in \mathbb{R}^{e}\times\mathbb{R}^{f}$$  such that the value $\mu_i$ appear exactly $r_i = \rk(E_i/E_{i+1})$ times in the first $e$ coordinates and the value $\mu_j'$ appears exactly $r'_j= \rk(F_j/F_{j+1})$ times in the last $f$ coordinates.

We rewrite the expressions of $R_a$ and $R'_b$ for $\alpha =(a,b)\in \Theta$ as follows. We have $$R_a\ =\rk(S_{A_a}E/S_{A_{a+1}}E)\hspace{4mm} = \sum\limits_{\substack{\underline{a} =(a_i)_{i=0}^{d-1}\in \mathbb{N}^d,\\ a_0+a_1+\cdots+a_{d-1}=mn, \\ a_0\mu_0+a_1\mu_1+\cdots+a_{d-1}\mu_{d-1} = v_a}} r_{\underline{a}} \hspace{5mm}$$

$$=\sum\limits_{\substack{\underline{a} =(a_i)_{i=0}^{d-1}\in \mathbb{N}^d,\\ a_0+a_1+\cdots+a_{d-1}=mn, \\ a_0\mu_0+a_1\mu_1+\cdots+a_{d-1}\mu_{d-1} = v_a}} \prod\limits_{i=0}^{d-1}\binom{a_i+r_i-1}{{r_i}-1} \hspace{3mm} = \sum\limits_{\substack{\underline{b}=(b_i)_{i=1}^e \in \mathbb{N}^e,\\ b_1+\cdots+b_e=mn,\\ b_1s_1+b_2s_2+\cdots+b_es_e=v_a}}1$$
\vspace{1mm}

This follows from the fact that $\binom{a_i+r_i-1}{{r_i}-1}$ equals the number of compositions of $a_i$ into sum of $r_i$ positive integers.

Similarly $$R'_b =\rk(S_{A'_b}F/S_{A'_{b+1}}F)\hspace{5mm} = \sum\limits_{\substack{\underline{a}' =(a'_i)_{i=0}^{h-1}\in \mathbb{N}^h, \\ a'_0+a'_1+\cdots+a'_{h-1}=ln, \\ a'_0\mu'_0+a'_1\mu'_1+\cdots+a'_{h-1}\mu'_{h-1} = v'_b}} r_{\underline{a}'} \hspace{5mm} = \sum\limits_{\substack{\underline{b}'=(b'_i)_{i=1}^f \in \mathbb{N}^f,\\ b'_1+\cdots+b'_{f}=ln,\\ b'_1s'_1+b'_2s'_2+\cdot+b'_fs'_f=v'_b}}1$$

Thus
\begin{align*}
\nu_{\Sym^{mn}(E)\otimes \Sym^{ln}(F)}
& =  \frac{1}{\rk(\Sym^{mn}(E))\cdot\rk(\Sym^{ln}(F))} \sum\limits_{j=0}^{q-1}\mathcal{R}_j\delta_{w_j}
\\
\vspace{4mm}
& = \frac{1}{\rk(\Sym^{mn}(E))\cdot\rk(\Sym^{ln}(F))} \sum\limits_{j=0}^{q-1}\Bigl(\sum\limits_{v_a+v'_b = w_j} R_a\cdot R'_b\Bigr) \delta_{w_j}
\\
\end{align*}

Our claim is that $$\sum\limits_{j=0}^{q-1}\Bigl(\sum\limits_{v_a+v'_b = w_j} R_a\cdot R'_b\Bigr) \delta_{w_j} = \sum\limits_{\substack{\underline{b} = (b_i)_{i=1}^e \in \mathbb{N}^e, \underline{b}'= (b'_i)_{i=1}^f\in \mathbb{N}^{f},\\ \sum\limits_{i=1}^eb_i=mn,\sum\limits_{i=1}^fb'_i=ln}}\delta_{b_1s_1+b_2s_2+\cdots+b_es_e+b'_1s'_1+b'_2s'_2+\cdots+b'_fs'_f}
$$
Note that  $$\nu_{\Sym^{mn}(E)} = \frac{1}{\rk(\Sym^{mn}(E))}\sum\limits_{a=0}^{k-1}R_a\delta_{v_a} = \frac{1}{\rk(\Sym^{mn}(E))}\sum\limits_{\substack{\underline{b}=(b_i)_{i=1}^e \in \mathbb{N}^e, \\ \sum\limits_{i=1}^eb_i=mn }}\delta_{b_1s_1+b_2s_2+\cdots+b_es_e}$$
Similarly,  $$\nu_{\Sym^{ln}(F)} = \frac{1}{\rk(\Sym^{ln}(F))}\sum\limits_{b=0}^{p-1}R'_b\delta_{v'_b} = \frac{1}{\rk(\Sym^{ln}(F))}\sum\limits_{\substack{\underline{b}' = (b'_i)_{i=1}^f \in \mathbb{N}^f, \\ \sum\limits_{i=1}^fb'_i = ln }}\delta_{b'_1s'_1+b'_2s'_2+\cdots+b'_fs'_f}$$

 Now, our claim follows from this fact and the above observations.

Therefore,
\begin{align*}
& \nu_{\Sym^{mn}(E)\otimes \Sym^{ln}(F)}\\
& =  \frac{1}{\rk(\Sym^{mn}(E))\cdot\rk(\Sym^{ln}(F))} \sum\limits_{j=0}^{q-1}\mathcal{R}_j\delta_{w_j}
\\
\vspace{4mm}
& = \frac{1}{\rk(\Sym^{mn}(E))\cdot\rk(\Sym^{ln}(F))}\sum\limits_{\substack{\underline{b} = (b_i)_{i=1}^e \in \mathbb{N}^e, \underline{b}'= (b'_i)_{i=1}^f\in \mathbb{N}^{f},\\ \sum\limits_{i=1}^eb_i=mn,\sum\limits_{i=1}^fb'_i=ln}}\delta_{b_1s_1+b_2s_2+\cdots+b_es_e+b'_1s'_1+b'_2s'_2+\cdots+b'_fs'_f}
\end{align*}

Let $(X,\mathcal{L})$ and $(Y,\mathcal{G})$ be two measurable spaces. Let $\phi : X\longrightarrow Y $ be a measurable function. Then for any measure $\eta$ on $X$, we have the pushforward of the measure $\eta$ defined as follows:

 For any measurable set $A$ in $\mathcal{G}$, we have
 $$(\phi_*\eta)(A) = \eta(\phi^{-1}(A)).$$

Let $\Delta_m \times \Delta_l \subset \mathbb{R}^e\times \mathbb{R}^f$ be the set, where simplexes $\Delta_m$ and $\Delta_l$ are defined as follows:
$$\Delta_m = \Bigl\{(x_1,x_2,\cdots,x_e)\in\mathbb{R}^e \mid 0\leq x_i \leq m, \sum\limits_{i=1}^{e}x_i = m\Bigr\},$$
$$\Delta_l=\Bigl\{(x'_1,x'_2,\cdots,x'_f)\in \mathbb{R}^f\mid 0\leq x'_j \leq l,\sum\limits_{j=1}^fx'_j= l\Bigr\}.$$
Let $\phi_{\underline{s}} : \Delta_m\times \Delta_l\longrightarrow \mathbb{R}$ be the map such that  $(x_1,x_2,\cdots,x_e,x'_1,\cdots,x'_f)\mapsto\sum\limits_{i=1}^{e} s_ix_i+\sum\limits_{j=1}^{f}s'_jx'_j$.

Consider the product measure on the product space $\Delta_m\times \Delta_l$ and the pushforward of measure under this map $\phi_{\underline{s}}$.
\begin{thm}\label{thm6.3}
 Let $\eta$ be the Lebesgue measure on $\Delta_m\times \Delta_l$ normalized such that $\eta(\Delta_m\times \Delta_l) = m^e\cdot n^f$. Then  $T_{\frac{1}{n}}\nu_{\Sym^{mn}(E)\otimes \Sym^{ln}(F)}$ converges vaguely to $\phi_{\underline{s}*}\eta$, where $$\phi_{\underline{s}} : \Delta_m \times \Delta_l\longrightarrow \mathbb{R}$$ such that\hspace{3cm} $(x_1,x_2,\cdots,x_e,x'_1,\cdots,x'_f)\mapsto\sum\limits_{i=1}^{e} s_ix_i+\sum\limits_{j=1}^{f}s'_jx'_j$.

 Hence $$\nu^{\pi}_{\mathcal{O}_{\mathbb{P}(\mathcal{E})}(m) \otimes \pi_1^*\mathcal{O}_{\mathbb{P}(F)}(l)} = \phi_{\underline{s}*}\eta.$$
 \begin{proof}
  Note that
  \begin{align*}
  & T_{\frac{1}{n}}\nu_{\Sym^{mn}(E)\otimes \Sym^{ln}(F)}
  \\
  & =T_{\frac{1}{n}}\Bigl[\frac{1}{\rk(\Sym^{mn}E)\cdot\rk(\Sym^{ln}(F)}\sum\limits_{\substack{\underline{b} = (b_i)_{i=1}^e \in \mathbb{N}^e, \underline{b}'=(b'_i)_{i=1}^f\in \mathbb{N}^{f},\\ \sum\limits_{i=1}^eb_i=mn,\sum\limits_{i=1}^fb'_i=ln}}\delta_{b_1s_1+b_2s_2+\cdots+b_es_e+b'_1s'_1+b'_2s'_2+\cdots+b'_fs'_f}\Bigr]
  \\
  &=T_{\frac{1}{n}}\Bigl[\frac{1}{\rk(\Sym^{mn}E)\cdot\rk(\Sym^{ln}(F)}\sum\limits_{\substack{\underline{b} = (b_i)_{i=1}^e \in (\frac{1}{n}\mathbb{N}^{e}), \underline{b}'=(b')_{i=1}^{f}\in (\frac{1}{n}\mathbb{N}^{e}), \\ \sum\limits_{i=1}^eb_i=m,\sum\limits_{i=1}^fb'_i=l}}\delta_{b_1+b_2+\cdots+b_e+b'_1+b'_2+\cdots+b'_f}\Bigr].
  \\
  &=T_{\frac{1}{n}}\Bigl[\frac{1}{\rk(\Sym^{mn}E)\cdot\rk(\Sym^{ln}(F)}\sum\limits_{(\underline{b},\underline{b}')\in (\frac{1}{n}\mathbb{N}^{e}\cap \Delta_m)\times (\frac{1}{n}\mathbb{N}^f\cap\Delta_l)}\delta_{(\underline{b},\underline{b}')}\Bigr].
  \end{align*}
  Thus $T_{\frac{1}{n}}\nu_{\Sym^{mn}(E)\otimes \Sym^{ln}(F)}$ converges vaguely to $\phi_{\underline{s}*}\eta$.
 \end{proof}
\end{thm}
\begin{corl}\label{thm6.4}
Consider the fibre product diagram:
\begin{center}
 \begin{tikzcd}
\mathbb{P}(\mathcal{E}) = \mathbb{P}(F)\times_C \mathbb{P}(E) = \mathbb{P}(\psi^*E) \arrow[r, "\pi_2"] \arrow[d, "\pi_1"]
& \mathbb{P}(E)  \arrow[d,""]\\
 \mathbb{P}(F) \arrow[r, "\psi" ]
& C
\end{tikzcd}
\end{center}
 Let $L = \mathcal{O}_{\mathbb{P}(\mathcal{E})}(m) \otimes \pi_1^*\mathcal{O}_{\mathbb{P}(F)}(l)\otimes \pi_1^*\psi^*\mathcal{M}$ be a line bundle on $X = \mathbb{P}(F)\times_C \mathbb{P}(E)$ for some line bundle $\mathcal{M}$ of degree $a$ on $C$ such that $L_{K}$ is big. Then
 $$\vol_X(L) =\dim(X)\vol_{X_K}(L_K) \int (x+a)_+ \phi_{\underline{s}*}\eta (dx),$$
 where $x_+ = \max\{x,0\}$ and $\phi_{\underline{s}}$ is defined as in Theorem \ref{thm6.3}.
 \end{corl}
 \begin{proof}
By Theorem \ref{thm2.1} one has 
\begin{align*}
& \vol_X(L)
 = \vol_X\bigl(\mathcal{O}_{\mathbb{P}(\mathcal{E})}(m)\otimes
\pi_1^*\mathcal{O}_{\mathbb{P}(F)}(l)\otimes \pi_1^*\psi^*\mathcal{M}\bigr)
\\
& = \dim(X) \vol_{X_K}(L_K) \int x_+\nu^{\pi}_L(dx)
\\
& = \dim(X)\vol_{X_K}(L_K) \int (x+a)_+\nu^{\pi}_{\mathcal{O}_{\mathbb{P}(\mathcal{E})}(m) \otimes \pi_1^*\mathcal{O}_{\mathbb{P}(F)}(l)} (dx)
\\
& = \dim(X)\vol_{X_K}(L_K) \int (x+a)_+ \phi_{\underline{s}*}\eta (dx).
\end{align*}
\end{proof}
\begin{exm}
\rm Let $\pi : X = \mathbb{P}_C(E)\longrightarrow C$ be a ruled surface over a smooth complex projective curve $C$.  We recall from \cite{C} the computation of volume function of any line bundle $L$ on $X$ such that $L_K$ is big line bundle on $X_K$.
Let $L = \mathcal{O}_{\mathbb{P}(E)}(m)\otimes \pi^*M$ where $\deg(M) = c$. Then by [Corollary 1.3,\cite{C}] we have $$\vol_X\bigl(\mathcal{O}_{\mathbb{P}(E)}(m)\otimes \pi^*M\bigr) = 2m^2\int \bigl(x+\frac{c}{m}\bigr)_+\nu_{\mathcal{O}_{\mathbb{P}(E)}(1)}^{\pi}(dx) = 2m^2\int \bigl(x+\frac{c}{m}\bigr)_+\phi_{\underline{s}*}\eta (dx)$$ where $\phi_{\underline{s}}$ and $\eta$ are defined as follows:

Let $\Delta\subset \mathbb{R}^2$ be the simplex defined as $\bigl\{(x_1,x_2) \in \mathbb{R}^2\mid 0\leq x_j\leq 1, \sum\limits_{i=1}^2x_i =1\bigr\}$, and $\eta$ be the Lebesgue measure on $\mathbb{R}^2$ which is normalized , i.e. $\eta(\Delta) =1$. Let
$$0=E_2\subsetneq E_1\subsetneq E_0 = E$$ be the Harder-Narasimhan filtration of $E$ with succesive  slopes $\mu_1 > \mu_0$.

Consider the vector $\underline{s} = (\mu_0,\mu_1)\in \mathbb{R}^2$. Then
$$\phi_{\underline{s}} : \Delta \longrightarrow \mathbb{R}$$
defined as $$\phi_{\underline{s}}(x_1,x_2) = \mu_0x_1+\mu_1x_2$$

Then by [Theorem 1.2, \cite{C}] we have $\nu_{\mathcal{O}_{\mathbb{P}(E)}(1)}^{\pi} = \phi_{\underline{s}*}\eta$.
We consider the following two cases :
\begin{itemize}
 \item \bf Case 1 : \rm Let $E$ be a semistable vector bundle on $C$. In this case, $\nu^{\pi}_{\mathcal{O}_{\mathbb{P}(E)}(1)} = \delta_{\mu(E)}$.
Hence $$\Phi(m,c) := \vol_X\bigl(\mathcal{O}_{\mathbb{P}(E)}(m)\otimes \pi^*M\bigr) = 2m^2\int \bigl(x+\frac{c}{m}\bigr)_+\phi_{\underline{s}*}\eta(dx)= 2m^2\bigl(\mu(E)+\frac{c}{m}\bigr)_+ $$
Therefore $$\Phi(m,c) = \begin{cases} \deg(E)m^2+2cm  & ; \text{if} \hspace{2mm} c\geq -\mu(E)m \\
                     0  &  ;\text{elsewhere}
                     \end{cases}$$

\item \bf Case 2 : \rm In this case, we assume that $E$ is an unstable bundle of rank 2 having the Harder-Narasimhan filtration as follows:
$$0=E_2\subsetneq E_1\subsetneq E_0 = E$$
with succesive quotients $\mu_1 > \mu_0$.
Then $\phi_{\underline{s}*}\eta$ is the uniform distribution on the interval $[\mu_0,\mu_1]$.

Hence $$\Phi(m,c) : = \vol_X\bigl(\mathcal{O}_{\mathbb{P}(E)}(m)\otimes \pi^*M\bigr) = \frac{2m^2}{\mu_1-\mu_0}\int\limits_{\mu_0}^{\mu_1}\bigl(x+\frac{c}{m}\bigr)_+(dx)$$
$$= \begin{cases}m^2(\mu_1+\mu_0)+2cm   & ; \text{if} \hspace{2mm} c\geq -m\mu_0 \\
(m\mu_1+c)^2/(\mu_1-\mu_0) & ; \text{if} \hspace{2mm} -m\mu_1 \leq c \leq -m\mu_0 \\

                     0  &  ;\text{if} \hspace{2mm} c<-m\mu_1
                     \end{cases}$$
\end{itemize}
In both the cases, the volume function $\Phi(m,c)$ is a polynomial in $m$ and $c$.
\end{exm}
\begin{xrem}\label{exm6.6}
\rm Let $E$ and $F$ be two semistable vector bundles of rank $e$ and $f$ respectively on a smooth complex projective curve $C$. We consider the fiber product $\pi: X = \mathbb{P}_C(E)\times_C \mathbb{P}_C(F) \longrightarrow C$.
In this case, $X \simeq \mathbb{P}_{\mathbb{P}(F)}(\pi_1^*E)$ where $\pi_1 : \mathbb{P}(F)\longrightarrow C$ is the projection map and $\pi_1^*E$ is semistable with $c_2(\End(\pi_1^*E)) = 0$. Thus $\Nef^1\bigl(\mathbb{P}(E)\times_C \mathbb{P}(F)\bigr) = \overline{\Eff}^1\bigl(\mathbb{P}(E)\times_C \mathbb{P}(F)\bigr)$. Hence for any divisor class $L\in \overline{\Eff}^1\bigl(\mathbb{P}(E)\times_C \mathbb{P}(F)\bigr)$, we have $\vol_X(L) = L^{\dim(X)}$.

Let $L=\mathcal{O}_{\mathbb{P}(\mathcal{E})}(m)\otimes \pi_1^*\mathcal{O}_{\mathbb{P}(F)}(l)\otimes \pi^*M$ be a line bundle on $X$ such that $L_K$ is big on $X_K$ (so that $m>0, l>0$), and $M$ is  a line bundle on $C$ of degree $a$.

 Note that as both $E$ and $F$ are semistable vector bundles on $C$, every symmetric powers of $E$ and $F$ and their tensor products are also semistable vector bundles.

Thus
$T_{\frac{1}{n}}\nu_{\pi_*(\mathcal{O}_{\mathbb{P}(\mathcal{E})}(m)\otimes \pi_1^*\mathcal{O}_{\mathbb{P}(F)}(m))^{\otimes n}} = T_{\frac{1}{n}}\nu_{\Sym^{mn}(E)\otimes \Sym^{ln}(F)}$ converges vaguely to \\
$\nu^{\pi}_{\mathcal{O}_{\mathbb{P}(\mathcal{E})}(m)\otimes \pi_1^*\mathcal{O}_{\mathbb{P}(F)}(l)} = \nu_{\Sym^{m}(E)\otimes \Sym^{l}(F)}$ which is the Dirac measure $\delta_{m\mu(E)+l\mu(F)}$ concentrated at $m\mu(E)+l\mu(F)$.

Therefore
\begin{align*}
 & \Phi(m,l,a) := \vol_X(L)
 \\
 & = \dim(X)\vol_{X_K}(L_K)\int (x+a)_+ \nu^{\pi}_{\mathcal{O}_{\mathbb{P}(\mathcal{E})}(m)\otimes \pi_1^*\mathcal{O}_{\mathbb{P}(F)}(l)} (dx)
 \\
 &= \dim(X)\vol_{X_K}(L_K)\int (x+a)_+ \delta_{m\mu(E)+l\mu(F)}(dx)
 \\
 &=  \begin{cases}    \dim(X)\vol_{X_K}(L_K)\bigl(m\mu(E)+l\mu(F)+a\bigr)\hspace{4mm}; \text{if} \hspace{2mm} a\geq -m\mu(E)-l\mu(F) \\
                     0   \hspace{7.5cm}; \text{elsewhere}
                     \end{cases}
\end{align*}
Note that $\dim(X) = e+f-1$ and $\vol_{X_K}(L_K) = L_K^{e+f-2}$ which is a polynomial in $m$ and $l$. Thus $\vol_X(L)$ is a polynomial in $m$, $l$ and $a$.
\end{xrem}
We end this section with the following explicit computation of volume function on the fiber product of two ruled surfaces.
\begin{exm}
\rm Let $\pi : X = \mathbb{P}(E)\times_C\mathbb{P}(F) \longrightarrow C$ be fiber product of two ruled surfaces, and Let $L=\mathcal{O}_{\mathbb{P}(\mathcal{E})}(m)\otimes \pi_1^*\mathcal{O}_{\mathbb{P}(F)}(l)\otimes \pi^*M$ be a line bundle on $X$ such that $L_K$ is big on $X_K$ (so that $m>0, l>0$), and $M$ is  a line bundle on $C$ of degree $a$.
Consider the following cases :
\begin{itemize}
 \item \bf Case 1 : \rm Suppose both $E$ and $F$ are semistable rank 2 bundles. Then
 \begin{align*}
 & \Phi(m,l,a) := \vol_X(L)
 \\
 &=  \begin{cases}    6ml\bigl\{m\mu(E)+l\mu(F)+a\bigr\}\hspace{3mm}; \text{if} \hspace{2mm} a\geq -m\mu(E)-l\mu(F) \\
                     0   \hspace{4.8cm}; \text{elsewhere}
                     \end{cases}
\end{align*}
\item \bf Case 2 : \rm Suppose $F$ is semistable and $E$ is unstable having Harder-Narasimhan filtration $$0=E_2\subsetneq E_1 \subsetneq E_0 = E$$ with succesive slopes $\mu_1>\mu_0$.
Then
\begin{align*}
 & \Phi(m,l,a) := \vol_X(L)
 \\
 &=  \begin{cases}    3ml\bigl\{m(\mu_1+\mu_0)+2l\mu(F)+2a\bigr\}\hspace{6mm}; \text{if} \hspace{2mm} a\geq -m\mu_0-l\mu(F)\vspace{4mm} \\
  3l\bigl\{(m\mu_1+l\mu(F)+a)^2/(\mu_1-\mu_0)\bigr\}\hspace{2.5mm}; \text{if} \hspace{2mm} -m\mu_1-l\mu(F) \leq a \leq -m\mu_0-l\mu(F)\vspace{4mm}\\
                     0   \hspace{6.3cm}; \text{elsewhere}
                     \end{cases}
\end{align*}
\item \bf Case 3 : \rm Suppose both $E$ and $F$ are unstable. Let $$0=E_2\subsetneq E_1 \subsetneq E_0 = E$$ be the Harder-Narasimhan filtration with succesive slopes $\mu_1>\mu_0$.
Similarly, let $$0=F_2\subsetneq F_1 \subsetneq F_0 = F$$ be the Harder-Narasimhan filtration with succesive slopes $\mu'_1>\mu'_0$.

Then \begin{align*}
 & \Phi(m,l,a) := \vol_X(L)
 \\
 &=  \begin{cases}    3ml\bigl\{m(\mu_1+\mu_0)+l(\mu_1'+\mu'_0)+2a\bigr\}\hspace{6mm}; \text{if} \hspace{2mm} a\geq -m\mu_0-l\mu'_0\vspace{4mm} \\
  3l\bigl\{(m\mu_1+l\mu'_1+a)^2/(\mu_1-\mu_0)\bigr\}\hspace{13mm}; \text{if} \hspace{2mm} -m\mu_1-l\mu'_1 \leq a \leq -m\mu_0-l\mu'_0\vspace{4mm}\\
                     0   \hspace{6.9cm}; \text{elsewhere}
                     \end{cases}
\end{align*}
\end{itemize}
In all the cases, the volume function $\Phi(m,l,a) := \vol_X(L)$ is a polynomial in the variables $m,l,a$.

\end{exm}

\section{Acknowledgement}
We are grateful to the referee for a careful reading and the many suggestions which improved the paper.  The first author is partially supported   by SERB-NPDF fellowship (File no : PDF/2021/00028). The second author is partially supported by SERB SRG Grant SRG/2023/001006. The first author also like to thank IIT Bombay for its hospitality where the work is initiated.

\end{document}